\documentclass{daj}

\dajAUTHORdetails{%
  title = {On Cap Sets and the Group-theoretic Approach to Matrix Multiplication}, 
  author = {Jonah Blasiak, Thomas Church, Henry Cohn, Joshua A.\ Grochow, Eric Naslund, William F.\ Sawin, and Chris Umans},
  plaintextauthor = {Jonah Blasiak, Thomas Church, Henry Cohn, Joshua A. Grochow, Eric Naslund, William F. Sawin, Chris Umans},
  runningauthor = {J.~Blasiak, T.~Church, H.~Cohn, J.~A.~Grochow, E.~Naslund, W.~F.~Sawin, and C.~Umans},
}   

\dajEDITORdetails{%
   year={2017},
   number={3},
   received={17 August 2016},   
   revised={5 January 2017},    
   published={16 January 2017},  
   doi={10.19086/da.1245},       
}

\usepackage{amsmath,amsthm,amsfonts,amssymb,amscd}
\usepackage{enumitem}
\usepackage{hyperref}
\usepackage{mathtools}
\newcommand\arXiv[1]{arXiv:\href{http://arXiv.org/abs/#1}{#1}}
\newcommand{\doi}[1]{\href{http://dx.doi.org/#1}{doi:#1}}

\newtheorem{theorem}{Theorem}[section]
\newtheorem{lemma}[theorem]{Lemma}

\newtheorem{definition}[theorem]{Definition}
\newtheorem{proposition}[theorem]{Proposition}
\newtheorem{observation}[theorem]{Observation}

\newtheorem{maintheorems}{Theorem}
\newenvironment{maintheorem}[1]{\renewcommand\themaintheorems{#1}\maintheorems}{\endmaintheorems}

\theoremstyle{definition}
\newtheorem{remark}[theorem]{Remark}

\numberwithin{equation}{section}

\newcommand\R{{\mathbb{R}}}

\newcommand\F{{\mathbb{F}}}
\newcommand\Z{{\mathbb{Z}}}
\newcommand\Q{{\mathbb{Q}}}
\newcommand\N{{\mathbb{N}}}

\newcommand{\abs}[1]{\left\lvert#1\right\rvert}

\DeclareMathOperator{\Rank}{rank}

\newcommand{\iso}{\cong}
\newcommand{\coloneq}{\mathrel{\mathop:}\mkern-1.2mu=}

\DeclareMathOperator{\tensorrank}{tensor-rank}
\DeclareMathOperator{\slicerank}{slice-rank}
\DeclareMathOperator{\trianglerank}{triangle-rank}
\DeclareMathOperator{\support}{support}
\renewcommand{\a}{\mathbf{a}}
\renewcommand{\b}{\mathbf{b}}
\renewcommand{\c}{\mathbf{c}}
\newcommand{\x}{\mathbf{x}}
\newcommand{\y}{\mathbf{y}}
\newcommand{\z}{\mathbf{z}}
\DeclareMathOperator{\instability}{instability}
\DeclareMathOperator{\avg}{avg}
\DeclareMathOperator{\SL}{SL}
\DeclareMathOperator{\bound}{bd}
\newcommand{\E}{\mathbb{E}}

\DeclareMathOperator{\lcm}{lcm}

\begin{document}

\begin{frontmatter}[classification=text]

\author[Blasiak]{Jonah Blasiak\thanks{Supported by NSF grant DMS-14071174.}}
\author[Church]{Thomas Church\thanks{Supported by NSF grant
DMS-1350138, the Alfred P.\ Sloan Foundation, and the Frederick E.\ Terman
Fellowship.}}
\author[Cohn]{Henry Cohn}
\author[Grochow]{Joshua A.\ Grochow\thanks{Supported by a Santa Fe Institute Omidyar Fellowship.}}
\author[Naslund]{Eric Naslund\thanks{Supported by Ben Green's ERC Starting Grant 279438, Approximate Algebraic Structure and Applications.}}
\author[Sawin]{William F.\ Sawin\thanks{Supported
by NSF grant DGE-1148900, Dr.~Max R\"{o}ssler, the Walter Haefner Foundation,
and the ETH Z\"{u}rich Foundation.}}
\author[Umans]{Chris Umans\thanks{Supported by NSF grant CCF-1423544 and a
Simons Foundation Investigator grant.}}

\begin{abstract}
In 2003, Cohn and Umans described a framework for proving upper bounds on
the exponent $\omega$ of matrix multiplication by reducing matrix
multiplication to group algebra multiplication, and in 2005 Cohn,
Kleinberg, Szegedy, and Umans proposed specific conjectures for how to
obtain $\omega=2$. In this paper  we rule out obtaining $\omega=2$  in
this framework from abelian groups of bounded exponent. To do this we
bound the size of tricolored sum-free sets in such groups, extending the
breakthrough results of Croot, Lev, Pach, Ellenberg, and Gijswijt on cap
sets. As a byproduct of our proof, we show that a variant of tensor rank
due to Tao gives a quantitative understanding of the notion of unstable
tensor from geometric invariant theory.
\end{abstract}
\end{frontmatter}

\section{Introduction}

A \emph{cap set} is a subset of $\F_3^n$ containing no lines; equivalently,
if $u$, $v$, and $w$ belong to the set, then $u+v+w=0$ if and only if
$u=v=w$. In a remarkable pair of recent papers \cite{CLP,EG}, Croot, Lev, and
Pach~\cite{CLP} introduced a powerful new technique, which Ellenberg and
Gijswijt~\cite{EG} used to prove that cap sets in $\F_3^n$ are bounded in
size by $O(c^n)$ with $c<3$, thus settling a long-standing open question. The
results of Ellenberg and Gijswijt similarly bound the size of subsets of
$\F_p^n$ that contain no three-term arithmetic progressions, as well as
certain more general sum-free sets.

Via the connections established earlier by Alon, Shpilka, and Umans~\cite{ASU}, the
cap set bounds prove the Erd\H os--Szemer\'edi sunflower conjecture
\cite{ES} and disprove the Coppersmith--Winograd ``no three disjoint
equivoluminous subsets'' conjecture \cite{CW}, which was proposed as a means
to show that the exponent $\omega$ of matrix multiplication is $2$. Alon
\emph{et al}.\ also showed that a \emph{tricolored} version  of the cap set
bounds would disprove the ``strong Uniquely Solvable Puzzle (USP)''
conjecture of Cohn, Kleinberg, Szegedy, and Umans~\cite{CKSU}, which was
another proposed approach to prove $\omega =2$ in the group-theoretic
framework of Cohn and Umans~\cite{CU}.
The strong USP conjecture is situated in
the context of a broader family of conjectures from \cite{CKSU}, which are
all potential means to prove $\omega = 2$. These conjectures all assert the
existence of certain large ``simultaneous triple product property'' (STPP)
constructions. An STPP construction is a collection of triples of subsets
$A_i,B_i,C_i\subseteq H$ inside a group $H$, satisfying certain conditions
(see Definition~\ref{def:STPP}). The approach of \cite{CKSU} shows that when
$H$ is abelian, any STPP construction implies the inequality
\begin{equation} \label{eq:ASI-STPP-ab-intro}
\sum_i (\abs{A_i}\abs{B_i}\abs{C_i})^{\omega/3} \le \abs{H}.
\end{equation}
If the sets involved are large enough, this yields a nontrivial bound on
$\omega$, and \cite{CKSU} showed how a family of sufficiently large STPP
constructions would imply $\omega=2$.

This paper contains two main results. The first is a bound on the size of
tricolored sum-free sets in abelian groups of bounded exponent. A tricolored
sum-free set is a generalization of a three-term progression-free set (a set
whose elements satisfy $u+v=2w \iff u=v=w$) in which the elements $u$, $v$,
and $w$ range over three \emph{different} subsets (see
Definition~\ref{def:tri-sum-free}). The \emph{exponent} of a finite group is
the least common multiple of the orders of all its elements; if a finite
abelian group is generated by elements of order at most $m$, then its
exponent is at most $\lcm(1,2,\dots,m) = e^{m(1+o(1))}$.  (This asymptotic
formula for $\lcm(1,2,\dots,m)$ is a variant of the prime number theorem.
See, for example, Theorem~11 in
\cite{RosserSchoenfeldApproximateFormulasForPrimes} for a stronger bound on
$\psi(m) = \log \lcm(1,2,\dots,m)$.)

\begin{maintheorem}{A}
\label{thm:boundedexponent} There exists an absolute constant
$\varepsilon=\frac{1}{2}\log((2/3)2^{2/3}) = 0.02831\ldots$ such that
if $H$ is an abelian group generated by elements of order at most $m$, then
every tricolored sum-free set in $H$ has size at most $3\cdot
\abs{H}^{1-\frac{\varepsilon}{m}}$.
\end{maintheorem}

In the case when $H$ is a power of a single cyclic group $\Z/p^k\Z$, we obtain a sharper bound.

\begin{maintheorem}{A$'$}
\label{thm:Zqn} There exists an absolute constant
$\delta=2\varepsilon=  0.05663\ldots$ such that
if $H\iso (\Z/q\Z)^n$ for some prime power $q$, then
every tricolored sum-free set in $H$ has size at most $3\cdot
\abs{H}^{1-\frac{\delta}{\log q}}$.
\end{maintheorem}

These theorems extend the Ellenberg--Gijswijt~\cite{EG} result in two
directions: from progression-free sets to tricolored sum-free
sets\footnote{Noga Alon independently observed that the Ellenberg--Gijswijt
result extends to tricolored sum-free sets in $\F_p^n$.} and from abelian
groups of prime exponent to abelian groups of bounded exponent. In
particular, Theorem~\ref{thm:boundedexponent} disproves the strong USP
conjecture via \cite{ASU}. Note that when $H$ has prime exponent $p$, the bound in
Theorem~\ref{thm:Zqn} of $3\cdot
\abs{H}^{1-\frac{\delta}{\log p}}$ agrees with the bound for progression-free sets obtained in~\cite{EG},
but with a slightly worse exponent as we have stated it uniformly in $p$.
It remains to be seen whether
$\frac{\varepsilon}{m}$ can be improved to
$\frac{c}{\log m}$ for general abelian groups of bounded exponent.
A generalization of the Ellenberg--Gijswijt result that handles progression-free
sets in abelian $p$-groups of bounded exponent was also obtained by Petrov in~\cite{Petrov}.

In the proof of Theorem~\ref{thm:boundedexponent}, we study a variant of
tensor rank due to Tao \cite{T}, which we call slice rank, and we show how
slice rank is related to a quantitative version of the notion of instability
from geometric invariant theory. We show that functions with
low slice rank are unstable, and conversely prove that a quantitative bound on
instability yields a bound on slice rank of tensor powers (see
Section~\ref{sec:instability}).

Our second main result generalizes \cite{ASU} by showing
that every STPP construction yields a large tricolored sum-free set.
Together these two facts show that it is impossible to prove $\omega=2$
using sets satisfying the simultaneous triple product property in  abelian
groups of bounded exponent:

\begin{maintheorem}{B} \label{thm:noSTTP}
For every $\ell\in \N$, there is an $\varepsilon_\ell>0$ such that no STPP
construction in any abelian group of exponent at most $\ell$ is large enough
to yield a bound better than $\omega\leq 2+\varepsilon_\ell$ via the
inequality~\eqref{eq:ASI-STPP-ab-intro}.
\end{maintheorem}

Note that all of the current best bounds on $\omega$ via the group-theoretic
approach (in \cite{CKSU}), as well as the current best bounds on $\omega$
which use the Coppersmith--Winograd approach \cite{CW, DS, W, L}, yield STPP
constructions whose underlying group is $(\Z/m\Z)^n$ for $m$ fixed.

However, our results \emph{do not} rule out achieving $\omega=2$ by using
STPP constructions over abelian groups.  Specifically, when the group has a
large cyclic factor, it indeed contains a large sum-free subset and thus the
constraints analyzed here are irrelevant. Furthermore, one can use
non-abelian groups or even more general objects such as association schemes
\cite{CU2}. Thus, our results serve to focus the search for group-theoretic
constructions, and certainly do \emph{not} rule them out as an approach to
achieving $\omega=2$.

\section{The simultaneous triple product property}

Recall that the \emph{exponent of matrix multiplication} is defined as
\[
\omega = \inf \{c : \text{the tensor rank of $n \times n$ matrix multiplication is at most } O(n^c) \text{ as } n \to \infty\}.
\]
In 2003, Cohn and Umans~\cite{CU} described a framework for proving upper
bounds on $\omega$ by reducing matrix multiplication to group algebra
multiplication. This reduction is carried out by means of a triple of
subsets satisfying the \emph{triple product property}:

\begin{definition}[{\cite[Definition~2.1]{CU}}]
Subsets $A,B,C$ of a group $G$ satisfy the \emph{triple product property} if
\[abc = 1 \iff a = b = c = 1\] for all $a\in A^{-1}A$, $b\in B^{-1}B$, and $c\in
C^{-1}C$, where $S^{-1}S$ denotes $\{s^{-1}s' : s,s' \in S\}$ for $S \subseteq G$.
\end{definition}

Such a triple of subsets realizes $\langle \abs{A},\abs{B},\abs{C}\rangle$
inside the group algebra of $G$. (Here $\langle m,n,p\rangle$ denotes the
matrix multiplication tensor for multiplying an $m\times n$ matrix by an
$n\times p$ matrix.) From this, letting $d_1,d_2,\dots$ be the character
degrees of $G$ (i.e., the dimensions of its irreducible representations), we
obtain the inequality
\[
(\abs{A} \abs{B} \abs{C})^{\omega/3}\leq \sum_i d_i^\omega
\]
by bounding the rank of group algebra multiplication \cite[Theorem~4.1]{CU}.
This inequality yields an upper bound for $\omega$ when $G$ and $A,B,C$ are
chosen appropriately.

In the later paper \cite{CKSU} several concrete routes to proving $\omega =
2$ were proposed. These proposals seemingly go beyond the framework of the
triple product property in various different ways; however as described in
\cite[\S7]{CKSU}, all of these constructions can be uniformly described
using the triple product property as follows. Several
independent matrix multiplications are realized via the triple product
property in the group algebra of a wreath product $H \wr S_m = H^m \rtimes
S_m$ using certain well-chosen subsets of a group $H$. This general
formulation is captured by the simultaneous triple product property:

\begin{definition}[{\cite[Definition~5.1]{CKSU}}] \label{def:STPP}
An \emph{STPP construction} is a collection of triples of subsets $A_i, B_i,
C_i$ of a group $H$ satisfying the \emph{simultaneous triple product
property} (STPP), which states that
\begin{enumerate}
\item for each $i$ the sets $A_i, B_i, C_i$ satisfy the triple product
    property, and
\item setting $S_i = A_iB_i^{-1}$, $T_j = B_jC_j^{-1}$, and $U_k =
    C_kA_k^{-1}$, \[ s_it_ju_k = 1 \quad\Rightarrow\quad i = j = k \] for
    all $s_i \in S_i$, $t_j \in T_j$, and $u_k \in U_k$.
\end{enumerate}
\end{definition}

Equivalently, $A_i, B_i, C_i \subseteq H$ satisfy the STPP if for all
$i,j,k$ and $s\in A_k$, $s'\in A_i$, $t\in B_i$, $t'\in B_j$, $u\in C_j$,
and $u'\in C_k$, we have
\[
s^{-1}s't^{-1}t'u^{-1}u'=1\quad\iff \quad i=j=k,\ \ s=s',\ \ t=t',\ \ u=u'.
\]

An STPP construction in $H$ realizes the tensor $\bigoplus_i \langle
\abs{A_i}, \abs{B_i}, \abs{C_i} \rangle$ and, via the asymptotic sum
inequality~\cite{S} or the use of a wreath product \cite[\S7]{CKSU}, yields
the fundamental inequality
\begin{equation}
\label{eq:ASI-STPP}
\sum_i (\abs{A_i}\abs{B_i}\abs{C_i})^{\omega/3} \le \sum_i d_i^\omega,
\end{equation}
where $d_1,d_2,\dots$ are the character degrees of $H$. For the rest of this
paper we will take $H$ to be abelian (with additive notation), in which case
this bound becomes the inequality \eqref{eq:ASI-STPP-ab-intro} of the
introduction:
\[
\sum_i (\abs{A_i}\abs{B_i}\abs{C_i})^{\omega/3} \le \abs{H}.
\]
All of the analysis of STPP constructions in the framework of \cite{CKSU} is based
on this inequality.

No single STPP construction can achieve $\omega=2$ via \eqref{eq:ASI-STPP}
(more generally, see \cite{CW2}), so we must consider families of
constructions in groups of growing size. To simplify the notation we
typically do not index such families explicitly (i.e., we refer to a group
$H$ rather than, say, $\{H_\alpha\}_{\alpha\in A}$).

Any STPP construction satisfies some simple ``packing bound'' inequalities,
which reflect the fact that the sets $S_i$ must be disjoint from each other, as must $T_i$
and $U_i$.  This disjointness follows immediately from the second condition
in the definition of the simultaneous triple product property. Furthermore,
since the sets $A_i, B_i, C_i$ satisfy the triple product property we must
have $\abs{S_i} = \abs{A_i}\abs{B_i}$, $\abs{T_i} = \abs{B_i}\abs{C_i}$, and
$\abs{U_i} = \abs{C_i}\abs{A_i}$.  Together these give the packing bounds:
\[
\sum_i \abs{A_i}\abs{B_i} \le \abs{H}, \qquad
\sum_i \abs{B_i}\abs{C_i} \le \abs{H}, \quad\text{and}\quad
\sum_i \abs{C_i}\abs{A_i} \le \abs{H}.
\]

\begin{definition}
We say that a family of STPP constructions with $\abs{H} \to \infty$ \emph{meets
the packing bound} if
\[
\sum_i \abs{A_i}\abs{B_i} \ge \abs{H}^{1 - o(1)}, \qquad
\sum_i \abs{B_i}\abs{C_i} \ge \abs{H}^{1 - o(1)}, \quad\text{and}\quad
\sum_i \abs{C_i}\abs{A_i} \ge \abs{H}^{1 - o(1)}.
\]
\end{definition}

A key observation is that meeting the packing bound is necessary for
achieving $\omega=2$:

\begin{lemma} \label{lemma:packing}
Any family of STPP constructions that does \emph{not} meet the packing bound
cannot imply $\omega=2$ via the inequality \eqref{eq:ASI-STPP-ab-intro}.
\end{lemma}

\begin{proof}
In our usual notation, if $\sum_i \abs{A_i}\abs{B_i} \le \abs{H}^{1 -
3\varepsilon}$ for some fixed $\varepsilon > 0$, then
\begin{align*}
\sum_i (\abs{A_i}\abs{B_i}\abs{C_i})^{\omega/3} &\le \Big(\sum_i (\abs{A_i}\abs{B_i}\abs{C_i})^{2/3}\Big)^{\omega/2}\\
&=\bigg(\sum_i \left(\abs{A_i}\abs{B_i}\abs{H}^{2\varepsilon}\cdot\abs{B_i}\abs{C_i}\abs{H}^{-\varepsilon}\cdot\abs{C_i}\abs{A_i}\abs{H}^{-\varepsilon}\right)^{1/3}\bigg)^{\omega/2}\\
&\le \bigg(\frac{\sum_i \abs{A_i}\abs{B_i}\abs{H}^{2\varepsilon} + \sum_i \abs{B_i}\abs{C_i}\abs{H}^{-\varepsilon} + \sum_i \abs{C_i}\abs{A_i}\abs{H}^{-\varepsilon}}{3}\bigg)^{\omega/2}\\
&\le \big( \abs{H}^{1-\varepsilon}\big)^{\omega/2},
\end{align*}
and so the strongest bound that can be obtained from
\eqref{eq:ASI-STPP-ab-intro} is $\omega\leq 2\cdot \frac{1}{1-\varepsilon}$,
which is bounded strictly away from the hoped-for $\omega=2$. The same holds
if either $\sum_i \abs{B_i}\abs{C_i} \le \abs{H}^{1 - 3\varepsilon}$ or $\sum_i
\abs{C_i}\abs{A_i} \le \abs{H}^{1 - 3\varepsilon}$.
\end{proof}

Both the strong USP conjecture~\cite[Conjecture~3.4]{CKSU} and the ``two
families'' conjecture~\cite[Conjecture~4.7]{CKSU} would, if true, yield STPP
constructions that meet the packing bound and moreover prove $\omega = 2$.
However, the STPP constructions produced by the strong USP conjecture have
underlying group $H = \F_3^n$, while the groups in the two families conjecture
need not have bounded exponent. Thus although Theorem~\ref{thm:noSTTP} disproves the
strong USP conjecture, it addresses only very special cases of
the two families conjecture.

\section{STPP constructions and tricolored sum-free sets}
\label{secmain:STPP-mcs}

\begin{definition}
\label{def:tri-sum-free}
A \emph{tricolored sum-free set} in an abelian group $H$ is a
3-dimensional perfect matching $M \subseteq S \times T \times U$ on a triple
of sets $S,T,U \subseteq H$, such that
\[
s+t+u = 0\text{ for all }(s,t,u) \in M
\]
and
\[
s+t+u \ne 0\text{ for all }(s,t,u) \in (S \times T \times U) \setminus M.
\]
The \emph{cardinality} of a tricolored sum-free set is the cardinality of
$M$.
\end{definition}

By a perfect matching $M\subseteq S\times T \times U$ we mean a subset whose
projection onto each of the three factors is a bijection. In other words,
given $s\in S$ there exist unique $t\in T$ and $u\in U$ such that
$(s,t,u)\in M$, and similarly for the other factors. The cardinality of the matching $M$ is
therefore equal to $\abs{S}$,  $\abs{T}$, and $\abs{U}$. Note that any
3-dimensional matching $M \subseteq H^3$ (not necessarily perfect, i.e.,
replacing ``bijection'' above with ``injection'') uniquely determines three
sets $S,T,U \subseteq H$ such that $M$ is a 3-dimensional perfect matching
on $S \times T \times U$.

If $S \subseteq \F_p^n$ contains no three-term arithmetic progressions, then
$M = \{(s,s,-2s): s \in S\}$ is a tricolored sum-free set. Similarly, for
any nonzero $\alpha,\beta,\gamma\in \F_p$ with $\alpha+\beta+\gamma=0$, we
obtain a tricolored sum-free set $M=\{(\alpha s,\beta s,\gamma s):s \in S\}$
whenever $S$ avoids nontrivial solutions to $\alpha s+\beta t+\gamma u=0$.

In this section, we show how to obtain a tricolored sum-free set from any
STPP construction in an abelian group. This allows us to prove that
Theorem~\ref{thm:boundedexponent} implies Theorem~\ref{thm:noSTTP}; we then
prove Theorem~\ref{thm:boundedexponent} in
Section~\ref{sec:boundedexponent}.

\subsection{STPP constructions imply tricolored sum-free sets}
\label{sec:STPP-mcs}

We will first construct slightly weaker objects we call border tricolored
sum-free sets, which are motivated by the notion of combinatorial
degeneration in the theory of border rank (see Definition~15.29 in
\cite{BCS}).  We will then use border tricolored sum-free sets to construct
genuine tricolored sum-free sets.

\begin{definition}
\label{def:border-tri-sum-free} A \emph{border tricolored sum-free set} in an
abelian group $H$ is a 3-dimensional perfect matching $M \subseteq S \times T
\times U$ on a triple of sets $S,T,U \subseteq H$ together with functions
$\alpha \colon S \to \Z$, $\beta \colon T \to \Z$, and $\gamma \colon U \to
\Z$ such that
\[
s+t+u = 0 \quad\text{and}\quad \alpha(s)+\beta(t)+\gamma(u)=0
\]
for all $(s,t,u) \in M$, while
\[
s+t+u \ne 0 \quad\text{or}\quad \alpha(s)+\beta(t)+\gamma(u)>0
\]
for all $(s,t,u) \in (S \times T \times U) \setminus M$. The
\emph{cardinality} of a border tricolored sum-free set is the cardinality of
$M$, and its \emph{range} is the maximum of $|\alpha(s)|$, $|\beta(t)|$, and
$|\gamma(u)|$ over $s \in S$, $t \in T$, and $u \in U$.
\end{definition}

One can reformulate the definition as follows: the sets $\{(s,\alpha(s)): s
\in S\}$, $\{(t,\beta(t)) : t \in T\}$, and $\{(u,\gamma(u)): u \in U\}$ form
a tricolored sum-free set in $H \times \Z$ (under the obvious $3$-dimensional
matching extending $M$), and they satisfy the positivity condition that
$\alpha(s)+\beta(t)+\gamma(u) \ge 0$ whenever $s+t+u=0$.

One of our main new contributions in this paper is the following
construction:

\begin{theorem} \label{thm:main}
Let $A_i, B_i, C_i \subseteq H$ be an STPP construction in an abelian group
$H$. Then there is a border tricolored sum-free set in $H$ of cardinality at
least
\[\sum_i \frac{\abs{A_i}\abs{B_i}\abs{C_i}}{\abs{A_i} + \abs{B_i} + \abs{C_i}}.\]
\end{theorem}

\begin{proof}
As in Definition~\ref{def:STPP} (but using additive notation), we will set
$S_i=A_i-B_i=\{a-b:a\in A_i,b\in B_i\}$, and similarly $T_i=B_i-C_i$ and $U_i=C_i-A_i$.

We will first construct a matching $M$ out of smaller matchings $M_i$ on
subsets of $S_i$, $T_i$, and $U_i$.  This matching will not quite be a
tricolored sum-free set in general, but we will define functions $\alpha$,
$\beta$, and $\gamma$ that repair any problems with it, so that it becomes a
border tricolored sum-free set.

To begin, let $n_i=\abs{A_i}$, $m_i = \abs{B_i}$, and $p_i = \abs{C_i}$, and
identify $A_i, B_i, C_i$ with $[n_i], [m_i], [p_i]$, respectively, via
bijections $\alpha_i, \beta_i, \gamma_i$. Let $r_i$ be the most frequently
occurring value in the multiset
\[\{x+y+z : (x,y,z) \in [n_i] \times [m_i] \times [p_i]\}.\]
Define $M_i \subseteq H^3$ as
\[
M_i = \{(a-b,b-c,c-a) : \mbox{$a \in A_i$, $b \in
B_i$, $c \in C_i$ such that $\alpha_i(a) + \beta_i(b) + \gamma_i(c) = r_i$} \}.
\]
The size of $M_i$ is the number of times $r_i$ occurs in the multiset above.
The number of distinct elements of the multiset is $n_i+m_i+p_i-2$; we can ignore the 2 and bound this by $\abs{A_i} + \abs{B_i} +
\abs{C_i}$. Since $r_i$ was chosen
to be the most frequent value,
\[
\abs{M_i} \geq \frac{\abs{A_i}\abs{B_i}\abs{C_i}}{\abs{A_i} + \abs{B_i} +
\abs{C_i}}.
\]

Each $M_i$ is a $3$-dimensional perfect matching: given the first coordinate
$a-b$, the triple product property determines $a$ and $b$ from $a-b$, and
then $\alpha_i(a) + \beta_i(b) + \gamma_i(c) = r_i$ determines $c$, and the
same is true for the other two coordinates.  We combine these matchings by
setting $M = \bigcup_i M_i$. Note that the sets $M_i$ are disjoint from each
other (since the sets $S_i$ are, for example), so $\abs{M} = \sum_i
\abs{M_i}$. All that remains is to define the functions $\alpha,\beta,\gamma$
so as to obtain a border tricolored sum-free set.

Let $S'_i$, $T'_i$, $U'_i$ be the projections of $M_i$ onto the three factors
of $H^3$.  In other words,
\[
S'_i = \{a-b : \mbox{$a \in A_i$, $b \in B_i$, and there exists $c \in C_i$ such that $\alpha_i(a) + \beta_i(b) + \gamma_i(c) = r_i$}\}\subseteq S_i,
\]
and $T'_i$ and $U'_i$ can be expressed similarly.  Let $S = \bigcup_i S'_i$,
$T=\bigcup_i T'_i$, and $U=\bigcup_i U'_i$, so that $M$ is a perfect matching between
$S$, $T$, and $U$.

To complete the proof, we must analyze when $s+t+u=0$ with $s \in S$, $t\in
T$, and $u \in U$. First, note that the simultaneous triple product property
implies that if $s_i+t_j+u_k=0$ with $s_i \in S_i$, $t_j \in T_j$, and $u_k
\in U_k$, then $i=j=k$.  Thus, we cannot obtain a sum of zero from $S'_i$,
$T'_j$, and $U'_k$ unless $i=j=k$, so we can analyze each matching $M_i$
individually, without worrying about how they might interact with each other.
More specifically, if for each $i$ we can find functions on $S'_i$,
$T'_i$, and $U'_i$ that make $M_i$ into a border tricolored sum-free
set, then this yields functions on $S$, $T$, and $U$ that make $M$ into
a border tricolored sum-free set.

Now consider $(s, t, u) \in S_i' \times T_i' \times U_i'$, with $s = a-b' \in
A_i - B_i$, $t = b-c' \in B_i - C_i$, and $u = c-a' \in C_i - A_i$. Note that
$s$ determines $a$ and $b'$, $t$ determines $b$ and $c'$, and $u$ determines
$c$ and $a'$ since $A_i, B_i, C_i$ satisfy the triple product property.
Furthermore, the triple product property tells us that
\[
s + t + u = 0 \quad\text{implies}\quad a = a',\  b = b',\  c = c'.
\]
However, that is not enough to conclude that $(s,t,u) \in M_i$, because it is
not necessarily the case that $\alpha_i(a)+\beta_i(b)+\gamma_i(c)=r_i$. To
address this issue, we will define functions $\alpha$, $\beta$, and $\gamma$
on $S'_i$, $T'_i$, and $U'_i$, respectively, so that
\[
\alpha(s) + \beta(t) + \gamma(u) = \big(\alpha_i(a) + \beta_i(b) + \gamma_i(c) -
r_i\big)^2
\]
whenever $s+t+u=0$. Specifically, we can take
\begin{align*}
\alpha(s) &= \alpha_i(a)^2 + 2 \alpha_i(a)\beta_i(b') - 2 \alpha_i(a) r_i + r_i^2,\\
\beta(t) &= \beta_i(b)^2 + 2 \beta_i(b)\gamma_i(c') - 2 \beta_i(b) r_i, \quad\text{and}\\
\gamma(u) &= \gamma_i(c)^2 + 2 \gamma_i(c)\alpha_i(a') - 2 \gamma_i(c) r_i.\\
\end{align*}
(We simply expand the square and assign each term to one of
$\alpha(s),\beta(t),\gamma(u)$ so that $\alpha(s)$ depends only on $a$ and
$b'$, $\beta(t)$ depends only on $b$ and $c'$, and $\gamma(u)$ depends only
on $c$ and $a'$.)

By construction, $\alpha(s) + \beta(t) + \gamma(u) \ge 0$ whenever $s+t+u=0$,
and $\alpha(s) + \beta(t) + \gamma(u)=0$ exactly when $(s,t,u) \in M_i$.
Thus, we have constructed a border tricolored sum-free set, as desired.
\end{proof}

\begin{lemma} \label{lem:unborder}
Let $H$ be an abelian group in which there is a border tricolored sum-free
set of cardinality $|M|$ and range $t$.  Then for each natural number $N$,
there exists a tricolored sum-free set in $H^N$ of cardinality at least
$|M|^N/(2Nt+1)^3$.
\end{lemma}

In particular, as $N \to \infty$ there are tricolored sum-free sets in $H^N$
of cardinality $|M|^{(1-o(1))N}$.

\begin{proof}
We will use the notation from Definition~\ref{def:border-tri-sum-free} for
the border tricolored sum-free set in $H$: let $M$ be the perfect matching on
sets $S,T,U \subseteq H$, with functions $\alpha \colon S \to \Z$, $\beta
\colon T \to \Z$, and $\gamma \colon U \to \Z$.

Define $M' \subseteq (S^N \times T^N \times U^N)$ in the natural way, so that
$((s_1,\dots,s_N),(t_1,\dots,t_N),(u_1,\dots,u_N)) \in M'$ if and only if
$(s_i,t_i,u_i) \in M$ for all $i$, and define $\alpha(s_1,\dots,s_N) =
\alpha(s_1) + \dots + \alpha(s_N)$ and similarly for $\beta$ and $\gamma$.
This construction yields a border tricolored sum-free set of cardinality
$|M|^N$ and range $Nt$ in $H^N$.

To obtain a genuine tricolored sum-free set, we will shrink $M'$ by a small
amount.  Because the functions $\alpha,\beta,\gamma$ for $M'$ have range $Nt$
(and thus each take on at most $2Nt+1$ values), there exist integers
$\alpha^*,\beta^*,\gamma^*$ such that for at least a $1/(2Nt+1)^3$ fraction
of the triples $(s',t',u') \in M'$, we have $\alpha(s') = \alpha^*$,
$\beta(t') = \beta^*$, and $\gamma(u') = \gamma^*$.  Furthermore,
$\alpha^*+\beta^*+\gamma^*=0$ because $\alpha(s')+\beta(t')+\gamma(u')=0$
whenever $(s',t',u') \in M'$.

Let $M''$ be the subset of $M'$ consisting of these triples with $\alpha(s')
= \alpha^*$, $\beta(t') = \beta^*$, and $\gamma(u') = \gamma^*$, and let
$S'',T'',U''$ be the sets on which $M''$ is a perfect matching.  Then $|M''| \ge
|M'|/(2Nt+1)^3 =$ $|M|^N/(2Nt+1)^3$, and $M''$ is trivially a border tricolored
sum-free set.  Furthermore,
\[
\alpha(s'') + \beta(t'') + \gamma(u'') =
\alpha^* + \beta^* + \gamma^* = 0
\]
whenever $s'' \in S''$, $t'' \in T''$, and
$u'' \in U''$, by construction.  Because $\alpha(s'') + \beta(t'') +
\gamma(u'')$ vanishes identically, the functions $\alpha$, $\beta$, and
$\gamma$ serve no purpose in the definition of a border tricolored sum-free
set. Thus, $M''$ reduces to an actual tricolored sum-free set, as desired.
\end{proof}

To control the size of the tricolored sum-free sets resulting from
Theorem~\ref{thm:main} and Lemma~\ref{lem:unborder}, we will need the
following notion. We say that an STPP construction is  \emph{uniform} if
$\abs{A_i}$ is independent of $i$, as are $\abs{B_i}$ and $\abs{C_i}$ (note
that we do not require $\abs{A_i}=\abs{B_i}=\abs{C_i}$).

\begin{lemma}
\label{lem:square}  If there is a family of STPP constructions in abelian
groups $H$ meeting the packing bound, then there is a family of
\emph{uniform} STPP constructions in powers of $H$ meeting the packing
bound.
\end{lemma}

\begin{proof}
Let the original STPP construction consist of $n$ triples $A_i, B_i, C_i$ of
subsets of $H$ indexed by $i\in [n]$. Our new STPP construction will consist
of subsets of $H^{3N}$, where $N$ is a large number to be chosen later;
these subsets are indexed by triples $(u,v,w) \in [n]^N\times [n]^N\times
[n]^N$ and defined by
\begin{align*}
\widehat{A}_{u,v,w} & =  \prod_{\ell} A_{u_\ell} \times \prod_{\ell}
B_{v_\ell} \times \prod_{\ell} C_{w_\ell}, \\
\widehat{B}_{u,v,w} & =  \prod_{\ell} B_{u_\ell} \times \prod_{\ell}
C_{v_\ell} \times \prod_{\ell} A_{w_\ell}, \\
\widehat{C}_{u,v,w} & =  \prod_{\ell} C_{u_\ell} \times \prod_{\ell}
A_{v_\ell} \times \prod_{\ell} B_{w_\ell}.
\end{align*}
(The products here are cartesian products of sets.) It is not hard
to verify that these sets satisfy the STPP in $H^{3N}$ (see
\cite[Lemma~5.4]{CKSU}). The resulting STPP construction is not yet uniform,
but will become so below when we restrict the choices of $u$, $v$, and $w$.
We first argue that the STPP construction $\widehat{A}_{u,v,w},
\widehat{B}_{u,v,w}, \widehat{C}_{u,v,w}$ meets the packing bound if the
original sets $A_i,B_i,C_i$ did.

To check that this construction meets the packing bound, we observe that
\[
\Big(\sum_i \abs{A_i}\abs{B_i}\Big)^N \cdot \Big(\sum_i \abs{B_i}\abs{C_i}\Big)^N \cdot \Big(\sum_i \abs{C_i}\abs{A_i}\Big)^N \ge \abs{H}^{3N(1 - o(1))}
\]
because the original STPP construction meets the packing bound. Expanding
the left side gives
\[
\sum_u \prod_{\ell}
\abs{A_{u_\ell}}\abs{B_{u_\ell}} \cdot \sum_v \prod_{\ell}
\abs{B_{v_\ell}}\abs{C_{v_\ell}} \cdot \sum_w \prod_{\ell}
\abs{C_{w_\ell}}\abs{A_{w_\ell}} = \sum_{u,v,w} \prod_\ell \abs{A_{u_\ell}}\abs{B_{v_\ell}}\abs{C_{w_\ell}}\abs{B_{u_\ell}}\abs{C_{v_\ell}}\abs{A_{w_\ell}}.
\]
We have
\begin{equation} \label{eq:AhatBhat}
\big\lvert\widehat{A}_{u,v,w}\big\rvert\big\lvert\widehat{B}_{u,v,w}\big\rvert = \prod_{\ell}
\abs{A_{u_\ell}}\abs{B_{v_\ell}}\abs{C_{w_\ell}}\abs{B_{u_\ell}}\abs{C_{v_\ell}}\abs{A_{w_\ell}}
\end{equation}
and hence
\[
\sum_{u,v,w} \big\lvert\widehat{A}_{u,v,w}\big\rvert\big\lvert\widehat{B}_{u,v,w}\big\rvert \ge \abs{H}^{3N(1 - o(1))},
\]
as desired; the same also holds for
$\sum_{u,v,w} \big\lvert\widehat{B}_{u,v,w}\big\rvert\big\lvert\widehat{C}_{u,v,w}\big\rvert$ and
$\sum_{u,v,w} \big\lvert\widehat{A}_{u,v,w}\big\rvert\big\lvert\widehat{C}_{u,v,w}\big\rvert$.

To enforce uniformity, we restrict our attention to only certain choices of
$u$, $v$, and $w$ by observing that the cardinality
\[
\big\lvert\widehat{A}_{u,v,w}\big\rvert =  \prod_{\ell} \abs{A_{u_\ell}}
\abs{B_{v_\ell}} \abs{C_{w_\ell}}
\]
depends only on the \emph{distributions} of $u,v,w$ (where the
distribution of $u$ is the vector specifying the number of times each
element of $[n]$ occurs in $u$). The same is true for $\lvert\widehat{B}_{u,v,w}\rvert$ and $\lvert\widehat{C}_{u,v,w}\rvert$.
There are $\binom{N+n-1}{n-1}$ possible
distributions, but all we need is the crude upper bound $(N+1)^n$ from the
fact that each element of $[n]$ occurs between $0$ and $N$ times. It follows
that there is at least one triple $\mu_1, \mu_2, \mu_3$ of distributions for
which
\begin{equation} \label{eq:choosemu}
\sum_{u \sim \mu_1} \prod_{\ell}
\abs{A_{u_\ell}}\abs{B_{u_\ell}} \cdot \sum_{v \sim \mu_2} \prod_{\ell}
\abs{B_{v_\ell}}\abs{C_{v_\ell}} \cdot \sum_{w \sim \mu_3} \prod_{\ell}
\abs{C_{w_\ell}}\abs{A_{w_\ell}} \ge \frac{1}{(N+1)^{3n}} \abs{H}^{3N(1 - o(1))},
\end{equation}
where $u \sim \mu_1$ means $u$ has distribution $\mu_1$.

Restricting to only those sets $\widehat{A}_{u,v,w}$, $\widehat{B}_{u,v,w}$, and
$\widehat{C}_{u,v,w}$ with $u \sim \mu_1$, $v \sim \mu_2$, and $w \sim
\mu_3$ thus gives a uniform STPP construction.
Combining
\eqref{eq:AhatBhat} and \eqref{eq:choosemu} we obtain
\[
\sum_{u \sim \mu_1, v \sim \mu_2, w \sim \mu_3}
\big\lvert\widehat{A}_{u,v,w}\big\rvert\big\lvert\widehat{B}_{u,v,w}\big\rvert \ge
\frac{1}{(N+1)^{3n}} \abs{H}^{3N(1 - o(1))},
\]
which is again $\abs{H}^{3N(1 - o(1))}$ as long as $N$ is chosen sufficiently
large relative to $n$ and $\abs{H}$. The same holds for the other two
conditions in the packing bound.

Choosing $N$ and $\mu_1,\mu_2,\mu_3$ in this way thus yields a family of
uniform STPP constructions meeting the packing bound in powers of $H$, as
desired.
\end{proof}

\subsection{\texorpdfstring{Theorem~\ref{thm:boundedexponent}}{Theorem A} implies \texorpdfstring{Theorem~\ref{thm:noSTTP}}{Theorem B}}
\label{sec:conclusion}

We conclude this section by  proving that Theorem~\ref{thm:boundedexponent}
implies Theorem~\ref{thm:noSTTP}; we will then prove
Theorem~\ref{thm:boundedexponent} in Section~\ref{sec:boundedexponent}.

Fix $\ell\in \N$, and suppose that for each $\delta>0$ there is an STPP
construction in a group of exponent at most $\ell$ that proves $\omega \le
2+\delta$.  Choosing a sequence with $\delta$ tending to zero, we obtain a
family of STPP constructions that meets the packing bound by
Lemma~\ref{lemma:packing}.  Furthermore Lemma~\ref{lem:square} lets us
inflate these constructions to make them uniform, while still meeting the
packing bound, in powers of the original groups; in particular, the exponent
of all our groups is still at most $\ell$.

Consider one of these uniform STPP constructions in a group $H$ of exponent
at most $\ell$, thus generated by elements of order at most $\ell$. By
Theorem~\ref{thm:main} there exists a border tricolored sum-free set in $H$
of cardinality $\abs{M}=\sum_i \frac{\abs{A_i}\abs{B_i}\abs{C_i}}{\abs{A_i} +
\abs{B_i} + \abs{C_i}}$. Since $\abs{A_i}$, $\abs{B_i}$, and $\abs{C_i}$ are
each independent of $i$, without loss of generality we may assume that
$\abs{C_i}$ is the largest of these, in which case
\[
\abs{M}=\sum_i
\frac{\abs{A_i}\abs{B_i}\abs{C_i}}{\abs{A_i} + \abs{B_i} + \abs{C_i}} \ge
\frac{1}{3} \sum_i \abs{A_i} \abs{B_i} \ge \abs{H}^{1-o(1)}.
\]
Finally, Lemma~\ref{lem:unborder} converts the border tricolored sum-free set
to a genuine tricolored sum-free set, at the cost of raising $H$ to a high
power (which of course does not change the exponent of this group). This
contradicts Theorem~\ref{thm:boundedexponent}, which states that the
cardinality of any tricolored sum-free set in $H^N$ is at most
$3\abs{H}^{N(1-\frac{\varepsilon}{\ell})}$, and thus completes the proof of
Theorem~\ref{thm:noSTTP}.

\section{Tricolored sum-free sets in abelian groups of bounded exponent via rank and instability of tensors}
\label{sec:boundedexponent}

The remainder of the paper is devoted to the proof of
Theorems~\ref{thm:boundedexponent} and~\ref{thm:Zqn}.
Our approach builds on the
symmetric formulation presented by Tao \cite{T}, although an earlier version
of this paper contained a more direct application of the original methods
\cite{CLP,EG} to the $\F_p^n$ case.

We begin here by outlining the main ideas of the proof, giving a road map
for the remaining sections. In order to formulate the symmetric version, we
develop the notion of \emph{slice rank} (in Section~\ref{sec:rank}), which
is a weakening of tensor rank. For tricolored sum-free sets in $\F_p^n$, the
upper bound is then a consequence of the following three facts about slice
rank:
\begin{enumerate}
\item Just as the ordinary matrix rank of an $m \times m$ diagonal matrix
    is equal to $m$ if all the diagonal entries are nonzero, the same holds
    for the slice rank of an $m \times m \times m$ diagonal
    tensor (over any field; see Lemma~\ref{lem:diag} or
    \cite[Lemma 1]{T}).

\item A tricolored sum-free set of cardinality $m$ in a group $G$ implies
    that the multiplication tensor of the group algebra
    $\F G$ restricts to an $m \times m \times m$ diagonal tensor with nonzero diagonal entries.  It follows that
    the slice rank of the multiplication tensor of $\F G$ is at least the
    size of any tricolored sum-free set in  $G$ (see
    Proposition~\ref{pr:sumfree-slicerank}).

\item Over a field of characteristic $p$, the slice rank of the
    $\F_p^n$-multiplication tensor is at most the number of vectors in
    $[p]^n$ with coordinate sum at most $pn/3$, which is at most
    $(p-\varepsilon_p)^n$ for some $\varepsilon_p > 0$.  By a remarkable
    observation \cite{CLP,EG,T}, this upper bound on the slice rank follows almost
    immediately from the fact that the $(x,y,z)$ entry of the tensor in
    question equals $\delta_0(x + y + z)$, and the delta function
    $\delta_0(x+y+z)$ (indeed \emph{every} function of $x+y+z$) is expressible as an $n$-variate polynomial of degree
$(p-1)n$; see Observation~\ref{obs:triangle}.
\end{enumerate}
These three statements give an upper bound of $(p-\varepsilon_p)^n$ on the
cardinality of a tricolored sum-free set in $\F_p^n$.
Expressed a different way, the upper bound for $H = \F_p^n$ is $\abs{H}^{1 -
\alpha_p}$ for some constant $\alpha_p > 0$. In
Section~\ref{sec:boundedexp-proof},  we prove an upper bound of the same
form for tricolored sum-free sets in any group of the form $G = (\Z/p^k \Z)^n$ with $p$ prime, using a similar
argument but using binomial coefficients instead of
polynomials to describe the delta function.

We then extend this bound to \emph{any} abelian group $H$ of bounded
exponent by arguing that such a group must decompose as $H\iso G\times K$ where $G \iso (\Z/p^k \Z)^n$ and $\abs{G} \ge
\abs{H}^c$ for a constant $c>0$. Such a decomposition expresses the $H$-multiplication
tensor $D_H$ as a tensor product $D_G\otimes D_K$. We establish a simple but powerful property of slice rank: the
slice rank of $T \otimes T'$ is at most the slice rank of $T$ times the side
length of $T'$ (Proposition~\ref{pr:slicerankmult}). This allows us to
reduce to the $(\Z/p^k \Z)^n$ case, because the slice rank of $D_H$ is at most $\abs{G}^{1 - \alpha}\abs{K} = \abs{H}^{1 - \alpha'}$ for some
constant $\alpha' > 0$. As before, this bound on slice rank also bounds the cardinality of tricolored sum-free sets in $G$.

In the presentation below, we draw connections to \emph{(in)stability} of
tensors, a notion coming from geometric invariant theory (GIT). This broader
context seems powerful, and potentially useful beyond the results in this
paper. However, the reader who is interested only in the proofs of
Theorems~\ref{thm:boundedexponent} and~\ref{thm:noSTTP} can skip all of
Section~\ref{sec:instability} except Lemma~\ref{lem:diag} and
Proposition~\ref{pr:sumfree-slicerank} and skip Theorem~\ref{thm:unstable2} since we only need
the shaper bound on slice rank stated in Proposition~\ref{pr:triangleslice}, coming from triangle rank.

\subsection{Tensor rank and its variants}
\label{sec:rank}

Throughout this section, $X$, $Y$, and $Z$ will denote finite sets. A
function $F\colon X\times Y\to \F$ with values in a field $\F$ has an
unambiguous \emph{rank}; $\Rank(F)$ is the smallest $k$ for which we can
write $F(x,y)=\sum_{i=1}^k f_i(x)g_i(y)$, and this coincides with the rank
of the $\abs{X}\times \abs{Y}$ matrix described by $F$.

A common way to define the rank of a function $F\colon X\times Y\times Z\to
\F$ is \emph{tensor rank}: $\tensorrank(F)$ is the smallest $k$ for which we
can write
\[F(x,y,z)=\sum_{i=1}^k f_i(x)g_i(y)h_i(z).\]
In this paper, we make use of another notion of rank which we call \emph{slice rank}:
$\slicerank(F)$ is the smallest $k$ for which we can write
\begin{equation}
\label{eq:slicerank}
F(x,y,z)=\sum_{i=1}^a f_i(x,y)g_i(z)+\sum_{i=a+1}^b f_i(x,z)g_i(y)+\sum_{i=b+1}^k f_i(y,z)g_i(x).
\end{equation}
As far as we know, this notion of rank was first used by Tao in \cite{T};
here we take this study a bit further by establishing some  basic properties of slice rank (Section~\ref{sec:rank})
and connections to GIT (Section~\ref{sec:instability}).

Since any sum $\sum_i f_i(x)g_i(y)h_i(z)$ automatically fits the form
\eqref{eq:slicerank}, we always have
\[
\slicerank(F)\leq \tensorrank(F)\leq\abs{\support (F)}.
\]
Similarly, directly from the definition \eqref{eq:slicerank} we immediately conclude that
\begin{equation} \label{eq:slicerankmin}
\slicerank(F)\leq \min(\abs{X},\abs{Y},\abs{Z})
\end{equation}
and
\begin{equation} \label{eq:tensorrankmax}
\tensorrank(F)\leq \slicerank(F)\cdot \max(\abs{X},\abs{Y},\abs{Z}).
\end{equation}

\begin{remark} \label{rmk:slicevtens}
Slice rank and tensor rank can be quite different. For example, if $\abs{X}=N$ the function
$F\colon X \times X \times X \to \F$ given by $F(x,y,z) = \delta_{x,y}$ has
$\slicerank(F) = 1$ while $\tensorrank(F) = N$. (Note that this is the
largest possible separation by \eqref{eq:tensorrankmax}.)

As another example of the difference between slice and tensor rank,
every function $F\colon X \times X \times X \to \F$ has slice rank at most
$N$ (and this bound is sharp; see Lemma~\ref{lem:diag} or \cite[Lemma
1]{T}). However, tensor rank can be much larger: the
tensor rank of a generic\footnote{That is, the
set of functions $F$ with this tensor rank is non-empty and Zariski-open.}
$F$ is $\lceil \frac{N^3}{3N-2} \rceil\approx N^2/3$ if $N \ne 3$ and $\F$ is algebraically closed \cite{Lickteig}.
\end{remark}

Given functions $F\colon X'\times Y'\times Z'\to \F$ and $G\colon X''\times
Y''\times Z''\to \F$, set $X=X'\times X''$, $Y=Y'\times Y''$, and
$Z=Z'\times Z''$, and let $F\otimes G$ denote the function $F\otimes G\colon
X\times Y\times Z\to \F$ given by
\[
\big((x',x''),(y',y''),(z',z'')\big)\mapsto F(x',y',z')G(x'',y'',z'').
\]
The inequality $\tensorrank(F\otimes G)\leq
\tensorrank(F)\cdot\tensorrank(G)$ is well-known; it has the following two
variants, which bound the slice rank of a tensor product:

\begin{proposition}
\label{pr:slicerankmult} For any $F,G$ as above, $\slicerank(F\otimes G)\leq
\slicerank(F)\cdot \tensorrank(G)$ and \newline $\slicerank(F\otimes G)\leq
\slicerank(F)\cdot \max(\abs{X''},\abs{Y''},\abs{Z''})$.
\end{proposition}

\begin{proof}
Set $k=\slicerank(F)$ and choose functions $f'_i$ and $g'_i$  for $1\leq
i\leq k$ as in \eqref{eq:slicerank}. Similarly, set $\ell=\tensorrank(G)$
and write
$G(x'',y'',z'')=\sum_{j=1}^{\ell}\alpha_j(x'')\beta_j(y'')\gamma_j(z'')$. If
we then define
\begin{align*}
f_{ij}(x,y)&\coloneq f'_i(x',y')\alpha_j(x'')\beta_j(y''),\qquad g_{ij}(z)\coloneq g_i(z')\gamma_j(z'')&\text{for }1\leq i\leq a,&\ 1\leq j\leq \ell,\\
f_{ij}(x,z)&\coloneq f'_i(x',z')\alpha_j(x'')\gamma_j(z''),\qquad g_{ij}(y)\coloneq g_i(y')\beta_j(y'')&\text{for }a< i\leq b,&\ 1\leq j\leq \ell, \text{ and}\\
f_{ij}(y,z)&\coloneq f'_i(y',z')\beta_j(y'')\gamma_j(z''),\qquad g_{ij}(x)\coloneq g_i(x')\alpha_j(x'')&\text{for }b< i\leq k,&\ 1\leq j\leq \ell,
\end{align*}
then
\[
(F\otimes G)(x,y,z)=\sum_{\substack{1\leq i\leq a\\1\leq j\leq \ell}}
f_{ij}(x,y)g_{ij}(z) +\sum_{\substack{a< i\leq b\\1\leq j\leq \ell}}
f_{ij}(x,z)g_{ij}(y)+\sum_{\substack{b< i\leq k\\1\leq j\leq \ell}}
f_{ij}(y,z)g_{ij}(x).
\]
This demonstrates that $\slicerank(F\otimes G)\leq k\cdot \ell$, which
verifies the first claim. For the second claim, we instead define
\begin{align*}
f_{i\zeta}(x,y)&\coloneq f'_i(x',y')G(x'',y'',\zeta),\qquad \,\,g_{i\zeta}(z)\coloneq g_i(z')\delta_{\zeta}(z'')&\text{for }1\leq i\leq a,&\ \zeta\in Z'',\\
f_{i\psi}(x,z)&\coloneq f'_i(x',z')G(x'',\psi,z''),\qquad g_{i\psi}(y)\coloneq g_i(y')\delta_\psi(y'')&\text{for }a< i\leq b,&\ \psi\in Y'', \text{ and}\\
f_{i\xi}(y,z)&\coloneq f'_i(y',z')G(\xi,y'',z''),\qquad \  g_{i\xi}(x)\coloneq g_i(x')\delta_\xi(x'')&\text{for }b< i\leq k,&\ \xi\in X'',
\end{align*}
so that
\[
(F\otimes G)(x,y,z)=\sum_{\substack{1\leq i\leq a\\\zeta\in Z''}} f_{i\zeta}(x,y)g_{i\zeta}(z)+\sum_{\substack{a< i\leq b\\\psi\in Y''}} f_{i\psi}(x,z)g_{i\psi}(y)+\sum_{\substack{b< i\leq k\\\xi\in X''}} f_{i\xi}(y,z)g_{i\xi}(x).
\]
This shows that $\slicerank(F\otimes G)\leq k\cdot
\max(\abs{X''},\abs{Y''},\abs{Z''})$, which verifies the second claim.
\end{proof}

\subsection{Unstable tensors and slice rank}
\label{sec:instability} In this section we relate slice rank to the notion
of an unstable tensor from geometric invariant theory. We show that
functions with low slice rank are unstable and prove that a quantitative
bound on instability yields a bound on slice rank of tensor powers.

\begin{definition}
\label{def:unstable} A function $F\colon X\times Y\times Z\to \F$ is
\emph{unstable} if there exist
\begin{enumerate}
\item a basis $\{f_a\}$ for the functions $X\to \F$, and similarly bases $\{g_b\colon Y\to \F\}$ and $\{h_c\colon Z\to \F\}$,
\item weights $u_a,v_b,w_c\in \R$ with arithmetic means
    $u_{\avg},v_{\avg},w_{\avg}$, and
\item coefficients $r_{a,b,c}\in \F$ such that
\end{enumerate}
\begin{equation*}
F(x,y,z)=\sum_{\substack{u_a+v_b+w_c<\\u_{\avg}+v_{\avg}+w_{\avg}}}r_{a,b,c}f_a(x)g_b(y)h_c(z).
\end{equation*}
\end{definition}

This terminology comes from geometric invariant theory (when $\F$ is
algebraically closed). Consider the action of the group $G=\SL_{\abs{X}}
\times \SL_{\abs{Y}}\times \SL_{\abs{Z}}$ on the vector space of functions
$F\colon X\times Y\times Z\to \F$. A function $F$ is said to be unstable in
the sense of GIT if the zero function is contained in the Zariski closure of
the $G$-orbit of $F$, or equivalently if every $G$-invariant homogeneous
polynomial of positive degree vanishes on $F$. The Hilbert--Mumford
criterion \cite[\S1]{Mumford} is a concrete condition for a function to be
unstable: it says that a function $F$ is unstable in the sense of GIT if and
only if there exist functions and weights making $F$ unstable according to
Definition~\ref{def:unstable}.

However, geometric invariant theory only deals with algebraically closed
fields $\F$, which is why for general $\F$ we take
Definition~\ref{def:unstable} as the \emph{definition} of unstable. To make
our bounds explicit we will also need to introduce a \emph{quantitative}
version of this condition.

\begin{definition}
\label{def:instability} The \emph{instability} of a function $F\colon
X\times Y\times Z\to \F$ is the supremum of all $\varepsilon \geq 0$ such
that there exist bases $f_a,g_b,h_c$, nontrivial weights  $u_a,v_b,w_c\in
\R$, and coefficients $r_{a,b,c} \in \F$ such that
\begin{equation}
\label{eq:unstable}
F(x,y,z)=\sum_{u_a+v_b+w_c\leq R}r_{a,b,c}f_a(x)g_b(y)h_c(z)
\end{equation} where
\begin{equation}
\label{eq:unstableR}
R=(u_{\avg}+v_{\avg}+w_{\avg})-\varepsilon(u_{\max}-u_{\min}+v_{\max}-v_{\min}+w_{\max}-w_{\min}).
\end{equation}
\end{definition}

By ``nontrivial weights'' we mean that the weights $u_a$, $v_b$, and $w_c$
should not all be constant (since in that case our definition of $R$ becomes
degenerate). Note that a function $F$ is unstable if and only if
$\instability(F) > 0$. By convention, if for some $F$ there is no nonnegative
$\varepsilon$ satisfying the hypotheses of the definition, we define
$\instability(F) = -\infty$. For readers familiar with the terminology of
GIT, we remark that (assuming $\F$ is algebraically closed) the function $F$
is \emph{semi-stable} if and only if $\instability(F)\in \{-\infty,0\}$, and
$F$ is \emph{stable} if and only if $\instability(F)=-\infty$. For
consistency with this terminology, we are careful to use the term ``not
unstable'' where appropriate (which is not the same, in GIT, as ``stable'').

\begin{remark}[Notes on the definition of instability]\label{rem:instability}\phantom{x}\ \phantom{x}
\begin{enumerate}[wide]
\item Translating all the weights $u_a$ (or all the $v_b$, or all the
    $w_c$) by a constant translates the quantity $R$ of
    \eqref{eq:unstableR} by the same amount, so the condition
    $u_a+v_b+w_c\leq R$ is invariant under translation. Therefore  if we
    like we may assume that $u_{\avg}=v_{\avg}=w_{\avg}=0$, or that
    $u_{\min}=v_{\min}=w_{\min}=0$, without loss of generality.
 Similarly, scaling all of the weights $u_a$, $v_b$, $w_c$ by the same constant does
    not change the definition of instability nor the value of
    $\instability(F)$.
\item It does not matter whether we require the weights to lie in $\R$,
    $\Q$, or $\Z$. Indeed, the supremum defining instability could be taken
    over rational $\varepsilon \geq 0$ without affecting the definition. For any given rational $\varepsilon$,
    the inequalities relating the $u_a$, $v_b$, and $w_c$ are a system of
    homogeneous linear inequalities with rational coefficients, so they
    have rational solutions if and only if they have real solutions. We can
    then transform rational weights to integer weights by scaling. (In particular,
    this justifies our earlier claim that the Hilbert--Mumford criterion is
    equivalent to Definition~\ref{def:unstable} when over an algebraically
    closed field; the usual statement of the Hilbert--Mumford criterion
    would require integer weights with $u_{\avg}=v_{\avg}=w_{\avg}=0$.)
\item The definition of instability would be the same if the $f_a$, $g_b$,
    and $h_c$ were arbitrary functions not required to form bases, as long
    as $\abs{A} \leq \abs{X}$, $\abs{B} \leq \abs{Y}$, and $\abs{C}\leq \abs{Z}$, where $A$, $B$,
    $C$ denote index sets for the $f_a$,  $g_b$,  $h_c$, respectively.
    Indeed, if one function $f_a$ is a linear
    combination of previous functions $f_{a'}$ with $u_{a'} \leq u_{a}$,
    then the terms $f_a\otimes g_b\otimes h_c$ appearing in
    \eqref{eq:unstable} can be replaced with a linear combination of terms $f_{a'}\otimes
    g_b\otimes h_c$, still satisfying $u_{a'}+v_b+w_c\leq u_a+v_b+w_c \leq
    R$. Repeating this, we may assume that the set of functions
    $\{f_a\}_{a \in A'}$ (with $A' \subseteq A$) appearing is linearly
    independent. Then extend this set to a basis arbitrarily, giving
    $\abs{X}-\abs{A}$ of these new functions the weight  $u_{\avg}$ and giving
    $\abs{A}-\abs{A'}$ of them the weights  $\{u_a\}_{a \in A \setminus A'}$. The
    other tensor factors are handled similarly.
\end{enumerate}
\end{remark}

Recall from \eqref{eq:slicerankmin} that any  function $F\colon X\times Y\times Z\to \F$ has
$\slicerank(F)\leq \min(\abs{X},\abs{Y},\abs{Z})$. It turns out that there
is a close relationship between $F$ being unstable and this inequality being
\emph{strict}. More precisely, the latter implies the former
(Theorem~\ref{thm:unstable1}), and the former implies that the latter holds
for sufficiently large tensor powers when $\abs{X}=\abs{Y}=\abs{Z}$
(Theorem~\ref{thm:unstable2}).

\begin{theorem}
\label{thm:unstable1}
If $\slicerank(F)<\min(\abs{X},\abs{Y},\abs{Z})$, then $F$ is unstable.
\end{theorem}

\begin{proof}
Given an arbitrary $F$, choose a decomposition
\begin{equation}
\label{eq:semistable-implies-full-wrank}
F(x,y,z)=\sum_{i=1}^p \alpha_i(y,z)f_i(x)+\sum_{j=1}^q\beta_j(x,z)g_j(y)+\sum_{k=1}^{r} \gamma_k(x,y)h_k(z)
\end{equation}
with $p+q+r=\slicerank(F)$. Note that this implies that the $p$ functions
$f_1,\dots,f_p$ are linearly independent (otherwise we could find a smaller
decomposition of $F$), so extend them to a basis $f_1,\dots,f_{\abs{X}}$
for the functions $X\to \F$. Do the same for the other factors, and define
weights by
\[u_i=\begin{cases}-1&\text{for }1\leq i\leq p,\\0&\text{for }p<i\leq \abs{X},\end{cases}\quad
v_j=\begin{cases}-1&\text{for }1\leq j\leq q,\\0&\text{for }q< j\leq \abs{Y},\end{cases}\quad
w_k=\begin{cases}-1&\text{for }1\leq k\leq r,\\0&\text{for }r< k\leq \abs{Z}.\end{cases}
\] Expanding
$\alpha_i(y,z)$ as a linear combination of $g_j(y)h_k(z)$ and so on, the
decomposition \eqref{eq:semistable-implies-full-wrank} says that $F$ is a
linear combination of functions $f_i\cdot g_j\cdot h_k$ where at least \emph{one}
of $1\leq i\leq p$, $1\leq j\leq q$, or $1\leq k\leq r$ holds.
From the definition of our weight functions, these are precisely the cases when $u_i+v_j+w_k \leq -1$.

On the other hand $u_{\avg}=-\frac{p}{\abs{X}}$,
$v_{\avg}=-\frac{q}{\abs{Y}}$, and $w_{\avg}=-\frac{r}{\abs{Z}}$, from which it follows that
\[
u_{\avg}+v_{\avg}+w_{\avg}= - \frac{p}{\abs{X}} - \frac{q}{\abs{Y}}- \frac{r}{\abs{Z}} \geq - \frac{ p + q + r}{\min(\abs{X},\abs{Y},\abs{Z})} =-\frac{\slicerank(F)}{\min(\abs{X},\abs{Y},\abs{Z})}. \]
Our hypothesis that $\slicerank(F)<\min(\abs{X},\abs{Y},\abs{Z})$ thus guarantees that these weights satisfy $u_{\avg}+v_{\avg}+w_{\avg}> -1$. In this case, $F$ is a linear combination of functions $f_i\cdot g_j\cdot h_k$ where
$u_i+v_j+w_k \leq -1< u_{\avg}+v_{\avg}+w_{\avg}$, so $F$ is unstable.
\end{proof}

Theorem~\ref{thm:unstable1} implies that any tensor $F$ that is not unstable
has
\[
\slicerank(F)=\min(\abs{X},\abs{Y},\abs{Z}).
\]
This lets us compute the slice rank of any diagonal. Say that a function
$F\colon X\times Y\times Z\to \F$ is a \emph{diagonal} if its support
$\{(x,y,z) : F(x,y,z)\neq 0\}$ is a perfect matching on $X$, $Y$, and $Z$.
Note that this implies $\abs{X}=\abs{Y}=\abs{Z}=\abs{\support(F)}$, and that
we require all the diagonal entries in $F$ (i.e., those corresponding to the perfect
matching) to be nonzero.

The following lemma was introduced to great effect by Tao \cite[Lemma~1]{T}.
We first reproduce Tao's elementary proof for completeness, and then include
a second proof inspired by the GIT perspective.

\begin{lemma}
\label{lem:diag} If $F\colon X\times Y\times Z\to \F$ is a diagonal, then
\[
\slicerank(F)=\tensorrank(F)=\abs{X}.
\]
\end{lemma}

\begin{proof}[Tao's proof from \cite{T}]
Suppose that $\slicerank(F)<|X|$. Then there exist $k<|X|$ and functions
$f_i, g_i$ such that
\begin{equation}
\label{equ:tao-proof-equation}
F(x,y,z)=\sum_{i=1}^{j}f_{i}(x)g_{i}(y,z)+\sum_{i=j+1}^{\ell}f_{i}(y)g_{i}(z,x)+\sum_{i=\ell+1}^{k}f_{i}(z)g_{i}(x,y).
\end{equation}
Without loss of generality suppose that $j>0$, and let
\[
V=\left\{ h\colon X\rightarrow\mathbb{F}:\ \sum_{x\in X}f_{i}(x)h(x)=0\text{ for all }1\leq i\leq j\right\}.
\]
This vector space has dimension at least $|X|-j$. Let $u\in V$ have maximal
support, and set $\Sigma=\left\{ x\in X:\ u(x)\neq0\right\}$. Then
$|\Sigma|\geq\dim V\geq|X|-j,$ since otherwise there exists nonzero $r\in V$
vanishing on $\Sigma,$ and the function $u+r$ would have a larger support
than $u.$ Multiply both sides of (\ref{equ:tao-proof-equation}) by $u(x)$ and
sum over $x$ to reduce from a tensor to a matrix. As $u\in V$ is orthogonal
to $f_i$ for $i=1,\dots,j$, the right side of (\ref{equ:tao-proof-equation})
becomes
\[
\sum_{i=j+1}^{\ell}f_{i}(y)\left(\sum_{x\in X}u(x)g_{i}(z,x)\right)+\sum_{i=\ell+1}^{k}f_{i}(z)\left(\sum_{x\in X}u(x)g_{i}(x,y)\right),
\]
which is a function of rank at most $k-j< |X|-j.$ Since the support of
$F(x,y,z)$ is precisely given by a perfect matching, $\sum_{x\in
X}u(x)F(x,y,z)$ will be a function with rank equal to
$|\support(u)|\geq|X|-j.$ Thus we have a contradiction, and it follows that
$\slicerank(F)=|X|.$
\end{proof}

We remark that Tao's proof is inductive, and extends to higher-order tensors
as well \cite{T}. The next proof of
Lemma~\ref{lem:diag} uses the geometric invariant theory perspective on
slice rank.

\begin{proof}[Proof of Lemma~\ref{lem:diag}]
By Theorem~\ref{thm:unstable1}, it suffices to prove that a diagonal is not
unstable.  It will be convenient to use dual bases to compute coefficients.
Given bases $f_a,g_b,h_c$ as in Definition~\ref{def:unstable},
consider the dual bases $f_a', g_b',h_c'$; in other words,
\[
\sum_{x \in X} f_a(x) f'_{a'}(x) = \begin{cases} 1 & \text{if $a=a'$, and}\\
0 & \text{otherwise,}
\end{cases}
\]
and the same holds for $g_b,g_b'$ (summing over $Y$) and $h_c,h_c'$ (summing over $Z$).
Our proof will depend on the following two observations.

First, the condition that
\[
F(x,y,z)=\sum_{\substack{u_a+v_b+w_c<\\u_{\avg}+v_{\avg}+w_{\avg}}}r_{a,b,c}f_a(x)g_b(y)h_c(z)
\]
is equivalent to the condition that
\begin{equation}
\label{eq:dualcond}
\sum_{x,y,z} F(x,y,z) f_{a}'(x)
g_{b}'(y) h_{c}'(z) \neq 0\quad\implies\quad u_a+v_b+w_c<u_{\avg}+v_{\avg}+w_{\avg},
\end{equation}
since that sum is precisely the coefficient $r_{a,b,c}$.

Second, the condition \eqref{eq:dualcond} is invariant under the change of basis that replaces $f_{a_2}'$ by $\alpha f_{a_1}'+f_{a_2}'$ as long as $u_{a_1} \geq u_{a_2}$. For if
\[\sum_{x,y,z} F(x,y,z)\big(\alpha f_{a_1}'(x)+ f_{a_2}'(x)\big)  g_{b}'(y) h_{c}'(z) \neq 0,\]
then either $\sum_{x,y,z} F(x,y,z)  f_{a_1}'(x)  g_{b}'(y) h_{c}'(z) \neq 0$
or $\sum_{x,y,z} F(x,y,z) f_{a_2}'(x)  g_{b}'(y) h_{c}'(z) \neq 0$, so
either $u_{a_1}+v_b+w_c$ or $u_{a_2}+v_b+w_c$ is less than
$u_{\avg}+v_{\avg}+w_{\avg}$. Thus in either case $u_{a_2}+v_b+w_c$ is less
than $ u_{\avg}+v_{\avg}+w_{\avg}$ as it is at most $u_{a_1}+v_b+w_c$. The
same is true for the $g_b$ and $h_c$.

We can now proceed to the proof.
For the sake of contradiction, assume that $F$ is a diagonal and unstable.
Without loss of generality, assume that $X=Y=Z=[n]$ and that the support of
$F$ is $\{(i,i,i) : i \in [n] \}$.

View the functions $f_a'\colon [n]\to \F$ as $n$-dimensional vectors. Using changes of variables of the aforementioned form, we may perform Gaussian
elimination until no two $f_a'$ have their first nonzero entries in the same
position. In that case, each element of $[n]$ is the position of the first
nonzero entry of some $f_a'$, so we may index the $f_a'$ by these positions.
Do the same for the $g_b'$, but for the $h_c'$, perform Gaussian elimination
so that the \emph{last} nonzero entries have distinct positions, and index the
$h_c'$ by these positions.

Because $F$ is diagonal,
\[
\sum_{x,y,z} F(x,y,z) f_{a}'(x) g_{b}'(y) h_{c}'(z) = \sum_{x} F(x,x,x)
f_a'(x) g_b'(x) h_c'(x)
\]
and $F(x,x,x) \neq 0$. Now for each index  $i \in [n]$, since $i$ is the
position of the first nonzero entry of $f_i'$ and $g_i'$ and the last
nonzero entry of $h_i'$,
\[ f_i'(x) g_i'(x) h_i'(x)=\begin{cases}f_i'(i) g_i'(i)h_i'(i) \neq 0&\text{if }x=i,\text{ and }\\0&\text{otherwise.}\end{cases}\]
This means
\[
\sum_{x} F(x,x,x) f_i'(x) g_i'(x) h_i'(x)= F(i,i,i) f_i'(i) g_i'(i)h_i'(i) \neq 0,
\]
which implies $u_i+v_i+w_i < u_{\avg} + v_{\avg}+w_{\avg}$. Summing over $i$, we have
\[\sum_{i=1}^n u_i + \sum_{i=1}^n v_i + \sum_{i=1}^n w_i < n u_{\avg} + n
v_{\avg} + n w _{\avg} = \sum_{i=1}^n u_i + \sum_{i=1}^n v_i + \sum_{i=1}^n
w_i, \] which is a contradiction, so $F$ is not unstable.
\end{proof}

Lemma~\ref{lem:diag} has the following important consequence. If $H$ is any
abelian group and $\F$ is any field, let $D_H\colon H\times H\times H\to \F$
be the tensor encoding the group structure (corresponding to multiplication
in the group algebra $\F H$), defined by
\[D_H(x,y,z)=\begin{cases}1&\text{if }x+y+z=0, \text{ and}\\0&\text{if }x+y+z\neq 0.\end{cases}\]

\begin{proposition}
\label{pr:sumfree-slicerank} If $M$ is a tricolored sum-free set in an
abelian group $H$, then
\[
{\abs{M}\leq \slicerank(D_H)}.
\]
\end{proposition}

\begin{proof}
When we restrict a function $F\colon X\times Y\times Z\to \F$ to subsets
$X_0,Y_0,Z_0$, we can also restrict any decomposition as in
\eqref{eq:slicerank}, showing that
\[\slicerank(F|_{X_0\times Y_0\times Z_0})\leq \slicerank(F).\]
A set $M\subseteq S\times T\times U$ on subsets $S,T,U\subseteq H$ is a
tricolored sum-free set if and only if the restriction $D_H|_{S\times
T\times U}$ is a diagonal; if it is, Lemma \ref{lem:diag} implies that
\[\abs{M}=\slicerank(D_H|_{S\times T\times U})\leq \slicerank(D_H).\qedhere\]
\end{proof}

This means we can bound the size of tricolored sum-free sets in an abelian
group $H$ by bounding the slice rank of $D_H$ from above.

\begin{remark}
With the notion of slice rank in hand, it is natural to wonder whether upper
bounds on slice rank could be applied directly to matrix multiplication to
get upper bounds on its tensor rank. Indeed, since $\tensorrank(\langle
n,n,n \rangle) \leq n^2 \cdot \slicerank(\langle n,n,n \rangle)$ by
\eqref{eq:tensorrankmax}, upper bounds of $O(n^{\delta})$ on the slice
rank of $n \times n$ matrix multiplication would imply $\omega \leq
2+\delta$. However, no nontrivial upper bounds on $\omega$ can be achieved
this way: it is known that the matrix multiplication tensor is not unstable
\cite[Theorem~5.2]{BI}, so Theorem~\ref{thm:unstable1} implies that
$\slicerank(\langle n,n,n \rangle) = n^2$. (However, the topic of the slice
rank of matrix multiplication itself is orthogonal to the main results of
this paper.)
\end{remark}

\subsection{Upper bounds on slice rank}
The following result gives us upper bounds on slice rank.

\begin{theorem}
\label{thm:unstable2} If $F$ is unstable, then for any $n\geq 1$,
\[\slicerank(F^{\otimes
n})\leq(\abs{X}^n+\abs{Y}^n+\abs{Z}^n)e^{-2n\instability(F)^2}.\]
\end{theorem}

We remark that in the most common case when $\abs{X}=\abs{Y}=\abs{Z}$, this
bound is always nontrivial for sufficiently large $n$. However if the sets
have unequal sizes, this only improves on the trivial bound
$\slicerank(F^{\otimes n})\leq\min(\abs{X}^n,\abs{Y}^n,\abs{Z}^n)$ if
$\instability(F)$ is sufficiently large.

\begin{proof}
It suffices to prove for all $\varepsilon<\instability(F)$ that
\[\slicerank(F^{\otimes
n})\leq(\abs{X}^n+\abs{Y}^n+\abs{Z}^n)e^{-2n\varepsilon^2}.\] Given such an
$\varepsilon$, we can choose functions $f_a,g_b,h_c$, weights $u_a,v_b,w_c$,
and coefficients $r_{a,b,c}$ as in Definition~\ref{def:instability} indexed
by $a\in A$, $b\in B$, and $c\in C$ such that
\[F(x,y,z)=\sum_{u_a+v_b+w_c\leq R}r_{a,b,c}f_a(x)g_b(y)h_c(z)\]
with $R$ as in \eqref{eq:unstableR}. Assume without loss of generality that
$u_{\min}=v_{\min}=w_{\min}=0$, and let us define bounds
$u_{\bound}=u_{\avg}-\varepsilon u_{\max}$, and similarly for $v_{\bound}$
and $w_{\bound}$, so that
\[R=u_{\avg}+v_{\avg}+w_{\avg}-\varepsilon(u_{\max}+v_{\max}+w_{\max})=u_{\bound}+v_{\bound}+w_{\bound}.\]

For $\a\in A^n$ and $\x\in X^n$, define $f_{\a}(\x)=\prod_{i=1}^n
f_{a_i}(x_i)$ and $u_{\a}=u_{a_1}+\cdots+u_{a_n}\in \R$; similarly define
$g_{\b}(\y)$, $v_\b$, $h_{\c}(\z)$, $w_{\c}$, and $r_{\a,\b,\c}=\prod_{i=1}^n
r_{a_i,b_i,c_i}$. Then \[F^{\otimes n}(\x,\y,\z)=\sum_{(\a,\b,\c)}
\prod_{i=1}^n
r_{a_i,b_i,c_i}f_{a_i}(x_i)g_{b_i}(y_i)h_{c_i}(z_i)=\sum_{(\a,\b,\c)}r_{\a,\b,\c}f_{\a}(\x)g_{\b}(\y)h_{\c}(\z),\]
where the sum is over triples $(\a,\b,\c)$ of tuples $\a\in A^n,\b\in
B^n,\c\in C^n$ satisfying the condition $u_{a_i}+v_{b_i}+w_{c_i}\leq
R=u_{\bound}+v_{\bound}+w_{\bound}$ for each $i=1,\dots,n$. This condition
implies $u_{\a}+v_{\b}+w_{\c}\leq nR=n(u_{\bound}+v_{\bound}+w_{\bound})$, so
each such triple must satisfy at least one of the inequalities $u_{\a}\leq
nu_{\bound}$, $v_{\b}\leq n v_{\bound}$, or $w_{\c}\leq n w_{\bound}$.
Therefore we can collect terms above to write
\[F^{\otimes n}(\x,\y,\z)=
\sum_{u_{\a}\leq nu_{\bound}}f_{\a}(\x)E_{\a}(\y,\z)+
\sum_{v_{\b}\leq nv_{\bound}}g_{\b}(\y)E_{\b}(\x,\z)+
\sum_{w_{\c}\leq nw_{\bound}}h_{\c}(\z)E_{\c}(\x,\y).\]

This decomposition exhibits a bound on the slice rank of $F^{\otimes n}$, so
all we need to do is estimate how many terms appear in these sums. The
number of tuples $\a\in A^n$ satisfying $u_{\a}\leq n u_{\bound}=n
(u_{\avg}-\varepsilon u_{\max})$ is a classical large deviation estimate:
$u_{\a}$ is the sum of i.i.d.\ variables $u_{a_i}$ taking values in
$[0,u_{\max}]$, and we are asking how often this deviates on one side from
the mean $nu_{\avg}$ by at least $\varepsilon nu_{\max}$. Even without
knowing anything about the distribution of $u_a$, Hoeffding's inequality
\cite[Theorem~1]{H} states that this proportion is bounded by
$e^{-2n\varepsilon^2}$, so the number of such tuples $\a\in A^n$ appearing in
the first sum is at most
$\abs{A^n}e^{-2n\varepsilon^2}=\abs{X}^ne^{-2n\varepsilon^2}$. The same bound
applies to the other two sums, so we conclude that $\slicerank(F^{\otimes
n})\leq (\abs{X}^n+\abs{Y}^n+\abs{Z}^n)e^{-2n\varepsilon^2}$.
\end{proof}

The bound in Theorem~\ref{thm:unstable2} can be improved if we have more
information about the distribution of the weights, by sharpening the large
deviation estimate in the last paragraph of the proof. (If $F$ is not
symmetric between $x$, $y$, and $z$, it may also be helpful to use different
cutoffs for the different variables.) One case is very common, so we single
it out: we define the \emph{triangle rank} $\trianglerank(F)$ to be the
smallest $k$ for which there exist functions $f_a$, $g_b$, $h_c$ for
$a,b,c\in\{0,\dots,k-1\}$ such that
\begin{equation}
\label{eq:trianglerankdef}F(x,y,z)=\sum_{a+b+c<k}r_{a,b,c}f_a(x)g_b(y)h_c(z).
\end{equation}
We remark that in contrast with the definition of instability, here we do
not require that the functions $f_i$ are linearly independent or form a
basis. However, it turns out we can assume this without affecting the
triangle rank by the same argument as Remark~\ref{rem:instability}(3).

We will not use the following observation directly, but it provides a
representative example and was central to Ellenberg--Gijswijt's results in
\cite{EG}.

\begin{observation} \label{obs:triangle}
For any function $P\colon \F_p\to \F_p$ with $p$ prime, the function
$F(x,y,z)=P(x+y+z)$ has triangle rank at most $p$.
\end{observation}

\begin{proof}
We can represent $P$ as a polynomial of degree less than $p$. Expanding
$P(x+y+z)$ and collecting terms expresses $F(x,y,z)=P(x+y+z)$ as a linear
combination of monomials $x^ay^bz^c$. Since $\deg P<p$, each monomial that
occurs satisfies $a+b+c<p$.
\end{proof}

To analyze the bounds on slice rank resulting from triangle rank, we will
need to bound the proportion of tuples $\a\in \{0,\dots,m\}^n$ with $\sum_i
a_i\leq \frac{1}{3}mn$. The bounds used in the proof of
Theorem~\ref{thm:unstable2} would bound this proportion by $e^{-n/18}$, but in this
case the rate function can be analyzed  more carefully, leading  to the
following definition. For $m>0$ and $\alpha\in (0,\frac{1}{2})$,  let
\[
I(m,\alpha)\coloneq \sup_{\theta<0}  \left(\alpha \theta-\log\left(\frac{1-e^{(1+1/m)\theta}}{(m+1)(1-e^{\theta/m})}\right)\right).
\]

\begin{proposition}
\label{pr:Is-increasing} The proportion of tuples $\a\in \{0,\dots,m\}^n$
satisfying $\sum_i a_i\leq \alpha mn$ is at most $e^{-I(m,\alpha)n}$.
Moreover, for fixed $\alpha\in (0,\frac{1}{2})$ the function $I(m,\alpha)$
is positive, is increasing in $m$, and converges to
$\sup_{\theta<0}\big(\alpha\theta-\log\big(\frac{e^{\theta
-1}}{\theta}\big)\big)$ as $m\to \infty$.
\end{proposition}

We will prove this proposition below; although the first assertion in this
proposition is standard, we include the proof for completeness. Before
moving to the proof of the proposition, we show that bounds on the triangle rank lead to
bounds on the slice rank of tensor powers. In this case the relevant proportion
will be $\alpha=\frac{1}{3}$, so for $s>1$ let us define
\begin{equation}
J(s)\coloneq e^{-I(s-1,\frac{1}{3})}=\frac{1}{s}\inf_{0<x<1}\frac{1-x^{s}}{1-x}x^{-\frac{s-1}{3}}.
\end{equation}
(The latter expression is obtained from the definition of $I(s-1,\alpha)$ by setting
$x=e^{\theta/(s-1)}$.) We point out as a consequence of
Proposition~\ref{pr:Is-increasing} that for all $s>1$ the function $J(s)$
is decreasing, satisfies $J(s)<1$, and (with $z=e^{-\theta/3}$)
\begin{equation}
\lim_{s\to \infty} J(s)
=\inf_{z>1}\frac{z-z^{-2}}{3 \log z}
= 0.8414\ldots.
\end{equation}

\begin{proposition}
\label{pr:triangleslice} If $\abs{X}=\abs{Y}=\abs{Z}=k$ and
$\trianglerank(F)\leq k$, then not only is $\instability(F)\geq
\frac{1}{6}$, but moreover the slice rank of $F^{\otimes n}$ is at most
$3(kJ(k))^n$.
\end{proposition}

\begin{proof}
We first check the first claim. Express $F$ as in \eqref{eq:trianglerankdef}
with the $f_a$, $g_b$, and $h_c$ linearly independent. Set $u_a=a$ for all
$a=0,1,\dots,k-1$; if the $f_a$ are not a basis, extend them arbitrarily to a
basis and set $u_a=\frac{k-1}{2}=u_{\avg}$ for all $a\ge k$ (these terms will
not play any role in the decomposition). Doing the same for the other
factors exhibits $F$ as a sum as in \eqref{eq:unstable} with $R=k-1$. Since
$u_{\max}=v_{\max}=w_{\max}=R$ and $u_{\avg}=v_{\avg}=w_{\avg}=\frac{R}{2}$,
solving for $\varepsilon$ gives $\varepsilon=\frac{1}{6}$.

To bound the slice rank of $F^{\otimes n}$, simply follow the proof of
Theorem~\ref{thm:unstable2} up until the last paragraph. The relevant
estimate in this case is the number of tuples $\a\in \{0,\dots,k-1\}^n$
with $\sum_i a_i\leq \frac{1}{3}(k-1)n$. In place of the general bound of
$k^ne^{-n/18}$ obtained there, substitute the bound of
$k^ne^{-I(k-1,\frac{1}{3})n}=k^nJ(k)^n$ from
Proposition~\ref{pr:Is-increasing} to conclude that  $\slicerank(F^{\otimes
n})\leq 3 (kJ(k))^n$.
\end{proof}

\begin{proof}[Proof of Proposition~\ref{pr:Is-increasing}]
After dividing everything by $m$, the first claim in the proposition states
that if $X_1,\dots,X_n$ are independent copies of a random variable $X$
uniformly distributed on the $m+1$ values
$\{\frac{0}{m},\frac{1}{m},\dots,\frac{m}{m}\}$, then $\Pr\big(\sum_i X_i\leq
\alpha n\big)\leq e^{-I(m,\alpha)n}$. For any $\theta <0$, applying the order-reversing
transformation $x\mapsto e^{\theta x}$ shows that $\Pr\big(\sum_i
X_i\leq \alpha n\big)=\Pr\big(e^{\theta\sum_i X_i} \geq e^{\alpha n\theta}\big)$. By
Markov's inequality, the latter is at most $e^{-\alpha n \theta} \E
\big(e^{\theta \sum_i X_i}\big)=\big(e^{-\alpha \theta} \E \big(e^{\theta X}\big)\big)^n$. Setting
$q=e^{\theta/m}$, we have $\E \big(e^{\theta X}\big) =
(1+q+\cdots+q^m)/(m+1)=\frac{1-q^{m+1}}{1-q}\cdot\frac{1}{m+1}$, and this
concludes the proof of the first claim.

For the remaining claims, set $G(t,x) =
\frac{1-e^te^{tx}}{(1+1/x)(1-e^{tx})}$, and set $G(0,x)=1$ so that $G$ is
continuous at $0$. By definition $I(m,\alpha)$ is the supremum over
$\theta<0$ of $\gamma(\theta,m)=\alpha\theta-\log G(\theta,\frac{1}{m})$.
Note that $\gamma(0,m)=0$; moreover, as $t \to 0$,
\[
\begin{split}
G(t,x) &= \frac{1-e^{t(1+x)}}{(1+1/x)(1-e^{tx})}\\
& = \frac{1}{1+1/x} \left(\frac{-t(1+x) - t^2(1+x)^2/2 + O(t^3)}{-tx - t^2x^2/2 + O(t^3)}\right)\\
& = \frac{1+t(1+x)/2 + O(t^2)}{1+tx/2+O(t^2)}\\
& = 1 + t/2 + O(t^2),
\end{split}
\]
and hence $\frac{\partial \log G}{\partial t}(0,x) = \frac{\frac{\partial
G}{\partial t}(0,x)}{G(0,x)} = \frac{\partial G}{\partial t}(0,x) =
\frac{1}{2}$. Thus, the derivative $\frac{\partial \gamma}{\partial
\theta}(0,m)=\alpha-\frac{1}{2}<0$ is negative for all $\alpha\in
(0,\frac{1}{2})$. This guarantees that there exists some $\theta<0$ where
$\gamma(\theta,m)$ is positive, proving the second claim.

Furthermore, the supremum of $\gamma(\theta,m)$ is attained at some
$\theta<0$, because $\gamma(\theta,m)$ tends to $0$ as $\theta \to 0$ and
$-\infty$ as $\theta \to -\infty$, while the supremum is positive. This means
that to prove that $I(m,\alpha)$ is increasing in $m$, it suffices to prove
for each \emph{fixed} $\theta<0$ that $\gamma(\theta,m)$ is increasing in
$m$. Equivalently, we must prove that $G(t,x)$ is increasing in $x$ when
$t<0$ and $x>0$; we do this by proving that $\frac{\partial G}{\partial x}>0$
in this region.

Set $h(x)=(1-e^{tx})/x$. Note that
$\frac{h(x+1)}{h(x)}=\frac{(1-e^{t(x+1)})/(x+1)}{(1-e^{tx})/x}=G(t,x)$.
Therefore $\frac{\partial G}{\partial x}>0$ if and only if $h'(x+1)h(x)\geq
h'(x)h(x+1)$. Since $t<0$ and $x>0$ we have $h(x)>0$, so this holds if and
only if $\frac{h'(x+1)}{h(x+1)}> \frac{h'(x)}{h(x)}$. Hence it suffices to
show that $\frac{d^2}{dx^2}\log h(x)>0$, since then $\frac{d}{dx}\log
h(x)=\frac{h'(x)}{h(x)}$ is increasing. We have that $\frac{d^2}{dx^2}\log
h(x)= \frac{d}{dx}\left( -\frac{1}{x} -  \frac{te^{tx} }{1- e^{tx}} \right) =
\frac{1}{x^2} - \frac{ (1-e^{tx}) t^2 e^{tx} - t e^{tx} (-t e^{tx})}{ (1-
e^{tx})^2} =\frac{1}{x^2}-\frac{t^2e^{tx}}{(1-e^{tx})^2}$, and so the
proposition follows from the inequality $(1-e^{tx})^2> t^2x^2e^{tx}$, which
we can prove as follows.

To verify this inequality, note that $n! \leq 2^{n-1} (n-1)!$ for all $n\in
\N$ with strict inequality for all $n>2$. Hence $\sum_{n=1}^\infty
\frac{z^n}{n!} > \sum_{n=1}^\infty \frac{z^n}{2^{n-1} (n-1)!}$ for $z>0$.
Evaluating those power series shows that $e^z-1 >  z e^{z/2}$. Dividing by
$e^z$, we obtain $1-e^{-z}>ze^{-z/2}$. Squaring yields
$(1-e^{-z})^2>z^2e^{-z}$; setting $z=-tx$, this becomes the desired
inequality.

Finally, as $m\to \infty$ we have
\[
\lim_{m\rightarrow\infty}I(m,\alpha)=
\lim_{m\rightarrow\infty}\sup_{\theta<0}\left(\alpha\theta-\log G\left(\theta,{\textstyle\frac{1}{m}}\right)\right).
\]
Since $-\log G(\theta,x)$ is monotonically decreasing in $x$,
$\alpha\theta-\log G\left(\theta,{\textstyle\frac{1}{m}}\right)$ is
monotonically increasing in $m$, so the limit in $m$ is the same as the
supremum in $m$ and hence commutes with the supremum in $\theta$. Thus we may
switch the order of the limit and the supremum to obtain
\[
\lim_{m\rightarrow\infty}I(m,\alpha)=\sup_{\theta<0}\big(\alpha\theta-\lim_{m\rightarrow\infty}\log G\left(\theta,{\textstyle\frac{1}{m}}\right)\big).
\]
By the continuity of the logarithm and the variable change $\frac{1}{m}=x$ this
equals
\[
\sup_{\theta<0}\left(\alpha\theta-\log\Big(\lim_{x\rightarrow0}{\textstyle\frac{h(x+1)}{h(x)}}\Big)\right),
\]
where $h(x)=\left(1-e^{\theta x}\right)/x$. Since
$\lim_{x\rightarrow0}h(x)=-\theta$, it follows that
\[
\lim_{m\rightarrow\infty}I(m,\alpha)=\sup_{\theta<0}\alpha\theta-\log\big({\textstyle \frac{e^{\theta }-1}{\theta}}\big).\qedhere
\]
\end{proof}

\subsection{Tricolored sum-free sets in abelian groups of bounded exponent}
\label{sec:boundedexp-proof}

In this section, we prove Theorems~\ref{thm:boundedexponent} and~\ref{thm:Zqn}, by way of the
following sharper bound.

\begin{theorem}
\label{thm:Zm} If $H$ is an abelian group $H\iso (\Z/q\Z)^n\times G$ where
$q$ is a prime power, then every tricolored sum-free set in $H$ has
cardinality at most $3\cdot \abs{H}\cdot J(q)^n$.
\end{theorem}

The upper bound in Theorem~\ref{thm:Zm} is known to be sharp up to a
subexponential factor by \cite[Theorem~2]{KSS}, conditional on a conjecture,
namely \cite[Conjecture~3]{KSS}, that was later proved independently in
\cite{Norin} and \cite{Pebody}.

The proofs of these theorems depend on the following proposition.

\begin{proposition}
\label{pr:power-of-p}
If $q=p^r$ is a prime power, the triangle rank of $D_{\Z/q\Z}$ over $\F_p$ is at most $q$.
\end{proposition}

\begin{proof}
The triangle rank is invariant under any permutation of the sets $X,Y, Z$,
so it suffices to bound the triangle rank of the function
\[
D_{\Z/q\Z}(x,y,z+1)=\begin{cases}1\in \F_p&\text{if }x+y+z=q-1\in \Z/q\Z, \text{ and}\\0\in \F_p&\text{if }x+y+z\neq q-1\in \Z/q\Z.\end{cases}
\]

As a consequence of Lucas' theorem, for any $0\leq k<q$ and $m,m'\in \N$,
\[
m\equiv m'\pmod{q}\quad\implies\quad \binom{m}{k}\equiv\binom{m'}{k}\pmod{p}.
\]
In other words, $m\mapsto \binom{m}{k}$ descends to a well-defined function
$\Z/q\Z\to \F_p$. We claim that
\[D_{\Z/q\Z}(x,y,z+1)=\sum_{a+b+c=q-1}\binom{x}{a}\binom{y}{b}\binom{z}{c}.\]
Indeed, the identity
$\binom{X+Y+Z}{i}=\sum_{a+b+c=i}\binom{X}{a}\binom{Y}{b}\binom{Z}{c}\in \N$
for $X,Y,Z\in \N$ descends to an identity
$\binom{x+y+z}{i}=\sum_{a+b+c=i}\binom{x}{a}\binom{y}{b}\binom{z}{c}\in \F_p$
for $x,y,z\in \Z/q\Z$. But for any $w\in \Z/q\Z$, choosing a lift in
$\{0,\dots,q-1\}$ shows that
\[\binom{w}{q-1} =\begin{cases}1\in \F_p&\text{if }w=q-1\in \Z/q\Z, \text{ and}\\0\in \F_p&\text{if }w\neq q-1\in
\Z/q\Z.
\end{cases}\]
This decomposition of $D_{\Z/q\Z}(x,y,z+1)$ shows that its triangle rank is
at most $q$.
\end{proof}

\begin{proof}[Proof of Theorem~\ref{thm:Zm}]
Let $q=p^r$, and consider the tensor $D_H$ over $\F_p$. The decomposition
$H\iso (\Z/q\Z)^n\times G$ induces a decomposition
$D_H=D_{(\Z/q\Z)^n}\otimes D_G$. Proposition~\ref{pr:slicerankmult} thus
implies
\[\slicerank(D_H)\leq \slicerank(D_{(\Z/q\Z)^n})\cdot \abs{G}.\]
We proved in Proposition~\ref{pr:power-of-p}  that
$\trianglerank(D_{\Z/q\Z})\leq q$. By Proposition~\ref{pr:triangleslice},
this implies that $\slicerank(D_{(\Z/q\Z)^n})\leq 3 (qJ(q))^n$. We conclude
\[\slicerank(D_H)\leq 3(qJ(q))^n\cdot \abs{G}=3\cdot \abs{H}\cdot J(q)^n.\]
The bound on tricolored sum-free sets in $H$ then follows from
Proposition~\ref{pr:sumfree-slicerank}.
\end{proof}

\begin{proof}[Proof of Theorem~\ref{thm:Zqn}]
If $q$ is a prime power,  Theorem~\ref{thm:Zm} states that  the size of
any tricolored sum-free set in $H\iso (\Z/q\Z)^{n}$ is at
most $3 \cdot\abs{H}\cdot J(q)^n$. To get a uniform bound we recall that
$J(q)$ is decreasing by Proposition~\ref{pr:Is-increasing},
so the worst case for this bound  occurs when the prime power $q$ is $2$. Therefore if we set
\[
\delta=I(1,{\textstyle\frac{1}{3}})=\sup_{\theta<0}\left({\textstyle\frac{\theta}{3}}-\log\left(\frac{e^{2\theta}-1}{2e^{\theta}-2}\right)\right)
=\log\big((2/3)2^{2/3}\big) = 0.05663\ldots,
\]
then $J(q)\leq J(2)=e^{-\delta}=q^{-\frac{\delta}{\log q}}$.
The size of any tricolored sum-free set in $H\iso (\Z/q\Z)^{n}$ is thus at most
\[ 3\cdot \abs{H}\cdot J(q)^n\leq 3\cdot \abs{H}\cdot J(2)^n=3 q^n e^{-\delta n}=3 \cdot \abs{H}^{1-\frac{\delta}{\log q}}.\qedhere\]
\end{proof}

\begin{proof}[Proof of Theorem~\ref{thm:boundedexponent}]
Let $q_1,\dots,q_r$ denote all the prime powers less than or equal to $m$.
By the Chinese remainder theorem, we can write $H\iso \prod_{i=1}^r
(\Z/q_i\Z)^{n_i}$. Let $\ell$ be the index for which $n_\ell$ is largest.
Writing $\sum_i n_i=N$, we know that $n_\ell\geq N/r$. Since  $\sum_i n_i \log
q_i=\log \abs{H}$ and $q_i\leq m$ we have $N\geq \frac{\log \abs{H}}{\log
m}$.

Next, we claim that $r<\frac{2m}{\log m}$ for all $m>1$. By \cite[(3.2) and
(3.36)]{RosserSchoenfeldApproximateFormulasForPrimes} we have that
$r<\frac{m}{\log m}+\frac{3m}{2(\log m)^2}+\frac{1.4260 m^{1/2}}{\log 2}$.
This bound is at most $\frac{2m}{\log m}$ as long as $m\geq 242$. For the
remaining $1<m<242$, one can simply check directly that the number of prime
powers less than $m$ is less than $\frac{2m}{\log m}$.

Combining these bounds on $N$ and $r$ shows that $n_\ell\geq \frac{N}{r}\geq
\frac{\log \abs{H}}{2m}$. By Theorem~\ref{thm:Zm}, this implies the size of
any tricolored sum-free set in $H\iso (\Z/q_\ell\Z)^{n_\ell}\times G$ is at
most
\[
3\cdot \abs{H}\cdot J(q_\ell)^{n_\ell}\leq 3\cdot \abs{H}\cdot J(q_\ell)^{\log \abs{H}/(2m)}=3\cdot \abs{H}^{1+\log J(q_{\ell})/(2m)}.
\]
As in the proof of Theorem~\ref{thm:Zqn} above, the worst case for this bound is when the prime power $q_\ell$ is $2$; since $J(s)$ is decreasing, we can bound $\log J(q_\ell)\leq \log J(2)=-\delta$. Setting
\[
\varepsilon=\delta/2
=\log\big((2/3)2^{2/3}\big)/ 2 = 0.02831\ldots,
\]
the bound from Theorem~\ref{thm:Zm} then implies the size of any tricolored
sum-free set in $H$ is at most $3\cdot \abs{H}^{1+\log J(q_{\ell})/(2m)}\leq
3\cdot \abs{H}^{1-\varepsilon/m}$.
\end{proof}

\section*{Acknowledgements}

We thank the AIM SQuaRE program, the Santa Fe Institute, and Microsoft
Research for hosting visits.

\bibliographystyle{amsplain}

\begin{dajauthors}
\begin{authorinfo}[Blasiak]
Jonah Blasiak\\
Department of Mathematics\\
Drexel University\\
Philadelphia, PA 19104\\
\texttt{jblasiak@gmail.com}
\end{authorinfo}
\begin{authorinfo}[Church]
Thomas Church\\
Department of Mathematics\\
Stanford University\\
450 Serra Mall\\
Stanford, CA 94305\\
\texttt{tfchurch@stanford.edu}
\end{authorinfo}
\begin{authorinfo}[Cohn]
Henry Cohn\\
Microsoft Research New England\\
One Memorial Drive\\
Cambridge, MA 02142\\
\texttt{cohn@microsoft.com}
\end{authorinfo}
\begin{authorinfo}[Grochow]
Joshua A.\ Grochow\\
University of Colorado at Boulder \\
Department of Computer Science \\
1111 Engineering Drive \\
Boulder, CO 80309 \\
and \\
Santa Fe Institute\\
1399 Hyde Park Rd.\\
Santa Fe, NM 87501\\
\texttt{jgrochow@santafe.edu}
\end{authorinfo}
\begin{authorinfo}[Naslund]
Eric Naslund\\
Mathematics Department\\
Princeton University\\
Fine Hall, Washington Road\\
Princeton, NJ 08544\\
\texttt{naslund@math.princeton.edu}
\end{authorinfo}
\begin{authorinfo}[Sawin]
William F.\ Sawin\\
ETH Institute for Theoretical Studies \\
ETH Z\"{u}rich \\
8092 Z\"{u}rich \\
\texttt{william.sawin@math.ethz.ch}
\end{authorinfo}
\begin{authorinfo}[Umans]
Chris Umans\\
Computing and Mathematical Sciences\\
Caltech\\
1200 E.\ California Blvd.\\
Pasadena, CA 91125\\
\texttt{umans@cms.caltech.edu}
\end{authorinfo}
\end{dajauthors}
\end{document}